\documentclass[11pt,a4paper]{article}

\usepackage{authblk}
\usepackage[utf8]{inputenc}

\usepackage[OT1]{fontenc}

\usepackage[english]{babel}

\usepackage[bookmarks,pdfpagelabels,linkbordercolor={0.1 0.8 0.4}]{hyperref}

\usepackage{amsfonts}
\usepackage{amsthm,amsmath}
\def\h{ {\cal H} }

\def\aa{ {\bf a} }
\def\bb{ {\bf b} }
\def\pp{ {\bf p} }

\def\ele{ {\cal L} }

\def\b{ {\cal B} }

\def\m{ {\cal M} }

\def\t{ {\cal T} }
\def\s{ {\cal S} }
\def\z{ {\cal Z} }

\def\k{ {\cal K} }

\def\l2a{L^2_a(\mathbb{D})}
\def\CC{\mathbb{C}}
\def\R{\mathbb{R}}
\def\ZZ{\mathbb{Z}}
\def\DD{\mathbb{D}}
\def\TT{\mathbb{T}}
\def\hc{H^\infty + C}
\def\gr{\mathrm{Gr}(\mathcal{H})}
\def\grp{\mathrm{Gr}_\pp(\mathcal{H})}

\newcommand{\PI}[2]{\left\langle #1 , #2 \right\rangle}
\newcommand\blfootnote[1]{%
  \begingroup
  \renewcommand\thefootnote{}\footnote{#1}%
  \addtocounter{footnote}{-1}%
  \endgroup
}

\newtheorem{teo}{Theorem}[section]
\newtheorem{prop}[teo]{Proposition}
\newtheorem{lem}[teo]{Lemma}
\newtheorem{coro}[teo]{Corollary}

\theoremstyle{definition}

\newtheorem{hyp}[teo]{Hypothesis}
\theoremstyle{remark}
\newtheorem{rem}[teo]{Remark}
\newtheorem{ejem}[teo]{Example}
\newtheorem{ejems}[teo]{Examples}

 \title{\mbox{Grassmann geometry of zero sets in reproducing kernel Hilbert spaces}}
 
 \author[1,3]{Esteban Andruchow}
 \author[2,3]{Eduardo Chiumiento}
 \author[1,3]{Alejandro Varela}

 \affil[1]{\small{Instituto de Ciencias,  Universidad Nacional de Gral. Sarmiento, Los Polvorines, Argentina  }}
 \affil[2]{\small{Centro de  Matem\'atica, Facultad de Ciencias Exactas, UNLP, La Plata, Argentina }}
 \affil[3]{\small{Instituto Argentino de Matem\'atica  ``Alberto P. Calder\'on'', CONICET, Buenos Aires, Argentina}}
 \date{}

\begin{document}
	
	\maketitle

\begin{abstract} 
	Let $\h$ be a reproducing kernel Hilbert space of functions on a set $X$. We study the problem of finding a minimal geodesic of the Grassmann manifold of $\h$ that joins two subspaces consisting of functions which vanish on given finite subsets of $X$. We establish a necessary and sufficient condition for existence and uniqueness of geodesics, and we then analyze it in examples. We discuss the relation of the geodesic distance with other known metrics when the mentioned finite subsets are singletons. We find estimates on the upper and lower eigenvalues of the unique self-adjoint operators which define the minimal geodesics, which can be made more precise when the underlying space is the Hardy space. Also for the Hardy space we discuss the existence of geodesics joining subspaces of functions vanishing on infinite subsets of the disk, and we investigate when the product of projections onto this type of subspaces is compact. 
\end{abstract}

\blfootnote{Email addresses: eandruch@ungs.edu.ar, eduardo@mate.unlp.edu.ar, avarela@ungs.edu.ar}

\bigskip



\section{Introduction}

Let $\h$ be an infinite dimensional complex Hilbert space. The Grassmann manifold $\gr$ is the set of all closed subspaces of $\h$, or equivalently, the set of all  bounded self-adjoint projections acting in $\h$. %Concerning its differential structure, it has the structure of an infinite dimensional homogeneous space of the unitary group, which is additionally a submanifold of the space of self-adjoint operators. More interesting, 
It has  the structure of an infinite dimensional manifold, with a linear connection and a Finsler metric, where the following results about its geodesics were proved. Given two subspaces $\s,\t \in \gr$, there exists a unique   minimal geodesic curve of  $\gr$  that joins them if and only if 
\begin{equation}\label{condition for geodesics}
\s\cap \t^\perp=\s^\perp\cap \t=\{0\}.
\end{equation}
If this is the case, there exists a unique self-adjoint operator $X=X_{\s,\t}$ acting in $\h$ such that
\begin{equation}\label{the self-adjoint operator}
X\s\subset \s^\perp , \ X\t\subset \t^\perp, \|X\|\le \pi/2 \hbox{ and } e^{iX}\s=\t.
\end{equation}
The geodesic is given by $\delta(t)=e^{itX}\s$. This geodesic has minimal length with respect to the following Finsler metric on the Grassmann manifold: for a smooth  curve $\s_t$, $t\in I$ of closed  subspaces of $\h$, let $P(t)=P_{\s_t}$ (the orthogonal projection onto $\s_t$), the length of the curve is measured by
$$
 \int_I \left\| \frac{d}{dt} P(t)\right\| dt,
$$
where $\|\ \, \|$ denotes the usual norm of operators. For instance, the length of the minimal geodesic mentioned above is $\|X\|$. References for these facts are \cite{pr, cpr, p-q, A14}.

The object of this paper is to apply these results to the case when $\h$ is a reproducing kernel Hilbert  space of functions on  a set $X$,   and the subspaces are sets of functions which vanish at  given   subsets of $X$. Namely, if $\aa=\{\, a_1 , \ldots , a_n \, \} \subseteq X$, take the subspace
% and $\bb=\{\, b_1, \ldots , b_m \, \}$ are disjoints subsets of  $\mathbb{D}$,  take the subspaces 
%\begin{equation}\label{zeros}
$$
\z_\aa=\{\, f\in\h \, :\,  f(a_j)=0, \, j=1, \ldots, n \,  \}.%, \, \, \, \,  \z_\bb=\{\, f\in\h \, : \,  f(b_j)=0,  \, j=1, \ldots , m \, \}.
$$
%\end{equation}
Similarly, let $\z_\bb$ be the subspace associated to a set  $\bb=\{\, b_1 , \ldots , b_m \, \}\subseteq X$.
Thus, we investigate  the existence of a minimal geodesic in $\gr$ joining the subspaces $\z_\aa$ and $\z_\bb$. We are mainly interested in the discussion of examples, where specific tools of each reproducing kernel Hilbert space are used to understand the geodesics of the Grassmann manifold between the special type of subspaces mentioned. Among the reproducing kernel Hilbert space considered, we put special emphasis in the Hardy space of the unit disk. Furthermore, this is the only reproducing kernel Hilbert space where we study the problem for infinite sets $\aa$ and $\bb$.

	Notice that given two subspaces it might be difficult to verify  condition (\ref{condition for geodesics}) in practice. For instance in  a previous work  \cite{acl18}, it was shown that this condition for  two shift-invariant subspaces turns out to be linked to the deep problem of injectivity of Toeplitz operators \cite{MP10}. 	Thus particular examples provided by functional spaces help   to understand abstract results on the structure of geodesics of the Grassmann manifold previously obtained by  Porta and Recht \cite{pr}, Corach, Porta and Recht \cite{cpr}, Kovarik \cite{kov} and the first author \cite{A14, p-q}. Also it is interesting to point out here that the Grassmann manifold  plays an essential role in the metric theory of general (infinite dimensional) homogeneous spaces arising on operator theory. Minimality of geodesics in these spaces was proved by constructing length-reducing maps onto Grassmann manifolds \cite{cpr2, DMR}.

%, and then using the information of the geodesics of the them.	

%Despite the aforementioned abstract characterization of the existence of minimal geodesics in the Grassmann manifold,

The outline and main results of this paper are as follows. In Section \ref{finite sets} we establish a necessary and sufficient condition for the existence of minimal geodesics between $\z_\aa$ and $\z_\bb$ in general reproducing kernel Hilbert spaces, whenever $\aa$ and $\bb$ are finite subsets of the set $X$ (Proposition \ref{first result}). Geodesics  exist if and only if the sets $\aa$ and $\bb$ have the same cardinality. Furthermore, uniqueness of geodesics is equivalent to have a non vanishing determinant in terms of the reproducing kernel. We then analyze this uniqueness condition in some classical examples of spaces of analytic functions. For the  Hardy space  of the disk $H^2$ the uniqueness condition always holds true, but for the Bergman space we give counterexamples as well as a sufficient condition (Theorem \ref{when is true for the Bergman}). Also it follows immediately that for a shift-invariant subspace $\h=\theta H^2$ the uniqueness condition is related to the zeros of the inner function $\theta$. For the Bargmann-Segal space  an easy equivalent condition to uniqueness can be written for the case of subsets $\aa, \, \bb \subseteq \CC$  with two points.

 In Section  \ref{singletons} we consider the case where the sets $\aa$ and $\bb$ are singletons. By means of the geodesic distance associated to our Finsler metric we define a metric $\Gamma$ in the set $X$.
In Theorem \ref{coincidencia} we show the relation between the metric $\Gamma$ and  other well-known metrics on $X$ such as the Skwarcy\'nski  or Kobayashi  metrics. The definition of these metrics goes back to the works \cite{K59, MPS85} for the Bergman kernel. More recently  they were introduced in general reproducing kernel Hilbert spaces \cite{ARSW11}.  Then we study conditions to guarantee that $(X,\Gamma)$ is a complete metric space. This is equivalent to have that $\grp:=\{ \z_a : a \in X \}$ is closed in $\gr$ with the usual operator norm (see Prop. \ref{closed}). We then investigate the action of Moebius transformations in the case of reproducing kernel Hilbert spaces of analytic functions on the disk (Corollary \ref{moebius action}).

Under the assumption of the existence of a unique a minimal geodesic between $\z_\aa$ and $\z_\bb$, in  Section \ref{norm inequalities} we prove estimates  on the norm and least eigenvalue of the self-adjoint operator $X_{\aa,\bb}:=X_{\z_\aa,\z_\bb}$ used to construct the geodesic, which is characterized by the above properties (\ref{the self-adjoint operator}). These results are given in Corollaries \ref{corolario42} and \ref{corolario43}. Further estimates can be obtained for the Hardy space by using classical results by Adamjan, Arov and Krein on singular values of Hankel operators, and the Takenaka-Malmquist-Walsh basis (see Subsection \ref{estimates Hardy}).

In Section \ref{infinite sets}  we consider in the Hardy space  the case of infinite sets $\aa, \, \bb \subseteq \DD$. Then these sets must satisfy Blaschke condition, and existence of geodesics is now related  to their sets of accumulation points $\lim \aa$ and $\lim \bb$ contained in the unit circle. Using results of Sarason and Lee \cite{leesarason}, we prove that there is always a geodesic joining $\z_\aa$ and $\z_\bb$ when $\lim \aa \not\subset \lim \bb$ and $\lim \bb \not\subset \lim \aa$ (Proposition \ref{different supports}). On the other hand, if $\lim \aa=\lim \bb =\{ 1\}$, then there are examples of both existence and non-existence of geodesics. Based on the existence of co-divisible infinite Blaschke products in the Sarason algebra and Koosis functions we  give examples where $\dim \z_\aa \cap \z_\bb^\perp=0$ and $\dim \z_\aa \cap \z_\bb^\perp=m$, for all values $0\leq m \leq \infty$ (Theorems \ref{equal supports1} and \ref{equal supports2}).  We then study when the product of two projections onto subspaces related to $\z_\aa$ and $\z_\bb$ are compact.

\section{Finite zero sets}\label{finite sets}

%\medskip

We  begin by summarizing basic facts on the geodesics of the Grassmann manifold in the following remark.

\begin{rem}\label{geometria de subespacios}
  Let $\b(\h)$ be the algebra of all bounded operators acting in $\h$. The Grassmann manifold $\gr$ is a complemented submanifold of $\b(\h)$. Its tangent space $(T \gr)_P$ at $P$ is given by
$$
(T \gr)_P=\{\, Y=iXP-iPX \, : \, X^*=X \, \},
$$
which consists of self-adjoint operators  which are co-diagonal with respect to $P$ (i.e. $PY P=(I-P)Y(I-P)=0$). Denote by $\b(\h)_{sa}$ the space of self-adjoint operators.
A natural projection $E_{P}:\b(\h)_{sa} \to (T \gr)_P$ is given by
$$
E_P(X)=P X(I-P)+(I-P)XP.
$$
This map induces a linear connection: if $X(t)$ is a tangent field along a curve $\alpha(t)\in \gr$,
$$
\frac{D X}{d t}=E_\alpha({X}).
$$
The geodesics of $\gr$ starting at $P$ with velocity $Y$ have the form $\delta(t)=e^{t\tilde{Y}}Pe^{-t\tilde{Y}}$, where
$\tilde{Y}=[Y,P]$ is antihermitian and co-diagonal with respect
to $P$.
As we already observed in the Introduction, the operator norm induces a  Finsler metric on $\gr$. This metric is neither smooth, nor convex. 
\begin{enumerate}
\item[1.] Let $P$, $Q$ be two orthogonal projections such that $\|P-Q\|<1$. Then there exists a unique operator $X=X^*$ with $\|X\|<\pi/2$ such that $Q=e^{iX}Pe^{-iX}$, which is co-diagonal with respect to both
to $P$ and $Q$. The curve
\begin{equation}\label{geo g}
\delta(t)=e^{itX}Pe^{-itX}
\end{equation}
 is the unique  geodesic of $\gr$ joining $P$ and $Q$
(up to reparametrization). Moreover, this geodesic has minimal length. The exponent $X$  is indeed an analytic function of $P$ and $Q$ (see \cite{pr}).
\item[2.]     There is a   geodesic (equivalently a minimal geodesic) in $\gr$ joining $P$ and $Q$ if and only if
\begin{equation}\label{inters otra vez}
\dim R(P)\cap N(Q)= \dim R(Q)\cap N(P).
\end{equation}
If both dimensions  are equal to zero, then
there exists a unique geodesic of minimal length in $\gr$ joining $P$ and $Q$. This geodesic has the same form as in (\ref{geo g}) for a (unique) self-adjoint operator $X$ satisfying $\| X\|\leq \pi/2$. In particular, note that there can be a unique minimizing geodesic even if $\|P-Q\|=1$. If the above dimensions coincide but are non zero, then there are infinitely many geodesics (see \cite{A14, p-q}).
\end{enumerate}
 The proof of condition (\ref{inters otra vez}) (also stated in (\ref{condition for geodesics})) about the existence of geodesics uses  Halmos' decomposition: given two subspaces $\s,\t$, one can decompose the underlying Hilbert space as
$$
\h=\s\cap \t  \ \oplus \ \s^\perp\cap \t^\perp  \ \oplus \ \s\cap \t^\perp \ \oplus\ \s^\perp\cap \t \ \oplus \ \h_0,
$$
where
$\h_0$ the orthogonal complement of the sum of the first four. The other first summands are usually denoted by $\h_{11}$, $\h_{00}$, $\h_{10}$ and $\h_{01}$, respectively.  
\end{rem}

Throughout this work,  $\h$ is a reproducing kernel Hilbert space consisting of functions on a set $X$. We refer to \cite{AMc02, paulsen} for theory of reproducing kernels Hilbert spaces.  For $w\in X$, denote by $k_w\in\h$ the reproducing kernel of $w$ defined by
$$
f(w)=\langle f,  k_w\rangle, \, \, \, \, f \in \h.
$$
%\item
%For any $S\in Gr(H)$, the map
%$$
%\pi_S:\u(H)\to Gr(H), \ \ \pi_S(U)=UP_S U^*,
%$$
%($\u(H)$= unitary group of $H$) whose range is the unitary orbit of $P_S$, which regarded again as subspaces is the orbit $\{U(S): U\in\u(H)\}$ of $S$ under the action of $\u(H)$, is a C$^\infty$ submersion. In particular, it has C$^\infty$
%local cross sections.

%Indeed, since $X$ is $P_0$-co-diagonal, it anticommutes with the symmetry $\epsilon_=\epsilon_{S_0}$:
%$$
%X\epsilon_0=-\epsilon_0X,
%$$
%and therefore
%$$
%\epsilon_t=2\delta(t)-1=e^{itX}\epsilon_0e^{-itX}=e^{2itX}\epsilon_0,
%$$
%and in particular $\epsilon_1=e^{2iX}\epsilon_0$. Thus
%$$

Let us  state the following elementary result for finite sets:

\begin{prop}\label{first result}
Let $\h$ be  a reproducing kernel Hilbert space of  functions on a set $X$. Take $\aa=\{a_1,\dots,a_m\}$, $\bb=\{b_1,\dots,b_k\}$  two  disjoint finite sets of points of $X$. 
Then the following conditions hold:
\begin{itemize}
\item[i)] There exists a geodesic in $\gr$ joining $\z_\aa$ and $\z_\bb$ if and only if $m=k$.
\item[ii)] If $k=m$, there is a unique geodesic if and only if 
\begin{equation}\label{det}
\det \left( ( k_{b_i}(a_j))_{1\le i,j\le k}\right) \ne 0.
\end{equation}
 Otherwise, there are infinitely many geodesics joining the aforementioned subspaces.  
\end{itemize}
\end{prop}
\begin{proof}
Throughout the proof, we shall use the matrix $K:= ( k_{b_i}(a_j))_{ij}$, $i=1, \ldots , k$, $j=1, \ldots, m$. Denote by $N(K)$ and $R(K)$ the nullspace  and range of $K$, respectively. We claim  that the map 
$$
\Phi: \z_\aa \cap \z_\bb^\perp \to N(K), \, \, \, \Phi\left(\sum_{i=1}^k \alpha_i k_{b_i}\right)=(\alpha_1, \ldots , \alpha_k) 
$$
is an isomorphism. Indeed, clearly  $\{ k_{b_1}, \, \ldots , k_{b_k} \}$ is a basis for $\z_\bb^\perp$. Thus, for each $f \in \z_\aa \cap \z_\bb^\perp$, there is a unique vector $(\alpha_1, \ldots , \alpha_k)  \in \CC^k$ such that
$f =\sum_{i=1}^k \alpha_i k_{b_i}$, which means that the above map is well defined. If $(\alpha_1, \ldots , \alpha_k)  \in N(K)$, then
the function $f=\sum_{i=1}^k \alpha_i k_{b_i}$ satisfies $f(a_j)=0$, $j=1, \ldots, m$, which says that  $f \in \z_\aa \cap \z_\bb^\perp$, and $\Phi(f)=(\alpha_1, \ldots , \alpha_k)$. Thus,
$\Phi$ is surjective. Clearly, it is also injective, so our claim is proved.  

\medskip

\noindent Proof  of $i)$:  It was proved in \cite[Thm 3.1]{A14} that there exists a geodesic joining $\z_\aa$ and $\z_\bb$ if and only if 
$$
\dim \z_\aa \cap  \z_\bb ^\perp =  \dim \z_\aa^\perp \cap  \z_\bb.
$$
Using that $\overline{{k_w}(z)}=k_z(w)$, for all $z,w \in X$, we have that $K^*=( k_{a_j}(b_i))_{ij}$, $j=1, \ldots , k$, $i=1, \ldots, m$. Interchanging the roles of the sets $\aa$ and $\bb$ in our previous claim, we also obtain  that $\z_\aa^\perp \cap  \z_\bb$ and $N(K^*)$ are isomorphic. Then the above condition to guarantee the existence of a geodesic holds true if and only if 
$\dim N(K)=\dim N(K^*)$.   Now the stated condition $m=k$ follows by using the well-known linear algebra formulas: $m=\dim N(K) + \dim R(K)$, $k=\dim R(K)^\perp + \dim R(K)$ and $ N(K^*)= R(K)^\perp$. 

\medskip

\noindent Proof of $ii)$: It was proved in \cite{A14, p-q} that there is  a unique geodesic if and only if $\z_\aa \cap  \z_\bb ^\perp =   \z_\aa^\perp \cap  \z_\bb =\{  0 \}$. This means that $N(K)=N(K^*)=\{ 0 \}$, which is equivalent to the stated condition on the determinant. 
\end{proof}

%---Version anterior al 27 agosto 2019 ------------------------------------------------------
%Let us state the following elementary Lemma:
%\begin{lem}
%Let  $\aa=\{a_1,\dots, a_k\}$ and $\bb=\{b_1,\dots, b_m\}$  be finite  subsets of $\DD$. Then
%$$
%\z_\aa \cap \z_\bb^\perp=\z_\aa^\perp\cap \z_\bb=\{0\}
%$$
%if and only if $k=m$
%\begin{equation}\label{det}
%\det \left( ( k_{b_i}(a_j))_{1\le i,j\le k}\right) \ne 0.
%\end{equation}
%\end{lem}
%\begin{proof}
%Clearly, the functions $k_{a_i}$, $1\le i \le k$ form a basis for $\z_\aa^\perp$.  Then $f\in\z_\aa^\perp\cap \z_\bb$ if and only if $f=\sum_{i=1}^k \alpha_i k_{a_i}$ and $f(b_j)=0$ for $1\le j \le m$, and
%$$
%\sum_{i=1}^k \alpha_i k_{a_i}(b_j)=0 , \ \ j=1,\dots, m.
%$$
%Therefore if $\z_\aa^\perp\cap \z_\bb$ is trivial, then $m\ge k$. Similarly, $\z_\bb^\perp\cap \z_\aa=\{0\}$ implies $k\ge m$, and both conditions mean that the determinant does not vanish. The converse  is clear.
%\end{proof}
%-------------------------------------------------------------------------------------------------------------

It follows that there exists a unique geodesic of the Grassmann manifold between $\z_\aa$ and $\z_\bb$ if and only if the determinant (\ref{det}) is non zero.
Suppose that $\aa$ and $\bb$ are two finite sets of the same cardinality. Suppose additionally that $\aa\cap\bb=\emptyset$.   
The set $\{k_w: w=a_i \hbox{ or } w=b_j\}$ is linearly independent. It follows that
$$
\z_\aa^\perp\cap\z_\bb^\perp=\{0\}.
$$
Therefore, in Halmos' decomposition of $\h$ in terms of $\z_\aa$ and $\z_\bb$, the only non trivial subspaces are
$\h_0$ and 
$$
\h_{11}=\z_{\aa\cup\bb}.
$$
Apparently, $\z_{\aa\cup\bb}$ has co-dimension $2k$. Also from these facts it is apparent that $k_{a_i}$ and $k_{b_j}$ belong to $\h_0$. It follows that $\h_0$ is generated by these functions.

\medskip
%\section{Condition (\ref{det}) in the examples}

Let us examine the determinant condition (\ref{det}) in some classical examples of spaces of analytic functions. 
Our examples consist in the Hardy space,  Bergman space, shift-invariant subspaces and Segal-Bargmann space. 

\medskip

\noindent \textbf{Hardy space}.
First we recall the definition of the Hardy space, and then we will show that condition (\ref{det}) always holds  in this space.
Let $L^2=L^2(\TT)$ denotes the usual Lebesgue space of complex valued functions defined on the unit circle $\TT$. The Hardy space $H^2=H^2(\DD)$ of the unit disk  $\DD=\{ \, z \in \CC \, : \, |z|<1 \, \}$  is the space of all analytic functions $f$  defined on $\DD$ for which
$$
\| f\|_{H^2}:= \left(\sup_{0<r<1}\frac{1}{2\pi} \int_0 ^{2\pi} |f(re^{it})|^2\, dt \right)^{1/2} < \infty.
$$
 Functions in the Hardy space have non tangential limits which can be used to  isometrically identify these spaces with
$$
H^2=\left\{ \, f \in L^2 \,  : \, \int_0^{2\pi} f(e^{it}) e^{-nit} \, dt=0, \, n >0 \, \right\}.
$$
In particular,  $H^2$ is a closed subspace of the Hilbert space $L^2$. It is a reproducing kernel Hilbert space, where the reproducing kernels are called 
Szego kernels:
$$
k_w^H(z)=\displaystyle{\frac{1}{1-\bar{w}z}} , \ z,w\in\mathbb{D}.
$$

\begin{prop}\label{hardy charact}
Let $\h=H^2$ be the Hardy space of the disk, then 
$$
\det((k^H_{b_j}(a_i))_{i,j})\ne 0,
$$
for every pair of disjoint subsets $\aa=\{a_1,\dots,a_n\}$ and $\bb=\{b_1,\dots,b_n\}$  of the disk $\mathbb{D}$.
\end{prop}
\begin{proof}
Instead of dealing directly with the determinant, we use the subspaces $\z_\aa$ and $\z_\bb$.
If the determinant is trivial, then $\z_\aa\cap\z_\bb^\perp$ is not trivial. Thus, there exists a function $f\in\z_\bb^\perp=\langle k_{b_j} : 1\le j \le n\rangle$ such that $f(a_i)=0$, $i=1,\dots,n$. Note that $f(z)=\sum_{j=1}^n \beta_j k_{b_j}(z)$ is a rational function which can be written
$$
f(z)=\displaystyle{\frac{p(z)}{\prod_{j=1}^n (1-\bar{b_j}z)}},
$$
where $p$  is a polynomial of at least degree $n-1$. This contradicts the fact that  $p$ has $n$ different roots $a_1,\dots, a_n$. 
\end{proof}

\begin{rem}
According to Propositions 2.2 and 2.3, in the Hardy space there are only two possibilities:  either there is no geodesic between $\z_\aa$ and $\z_\bb$, or there is a unique geodesic between these subspaces. The case of infinitely many geodesics described in Remark 2.1.2 cannot take place.  
\end{rem}

\medskip

\noindent \textbf{Shift-invariant subspaces}. As a straightforward consequence of the preceding example, we can now examine condition (\ref{det})  when the Hilbert space $\h$ is a shift-invariant subspace of $H^2$.  Therefore we have that $\h=\theta H^2$ for some inner function $\theta$, i.e. $\theta \in H^2$ and $|\theta(z)|=1$ a.e. on $\TT$. Then $\h$ has the reproducing kernel 
$$
k^{\theta}_w(z)=\frac{\theta(z)\overline{\theta(w)}}{1-z\bar{w}}, \, \, \, \, z, w \in \DD.
$$
Given $\aa=\{a_1, \ldots , a_n \}$ and $\bb=\{b_1, \ldots , b_n  \}$ subsets of $\DD$, note that
$$
\det((k^{\theta}_{b_j}(a_i))_{i,j})=\prod_{i=1}^n\theta(a_i)\overline{\theta(b_i)}\, \det((k^{H}_{b_j}(a_i))_{i,j}).
$$
Then using Proposition \ref{hardy charact},   there is a unique geodesic that joins the subspaces $\z_\aa$ and $\z_\bb$ in $\gr$ if and only if 
$\aa \cap \bb \cap \{ z \in \DD : \theta(z)=0 \}=\emptyset$. Otherwise, there are infinitely many geodesics joining the mentioned subspaces.

\medskip

\noindent \textbf{Bergman space.}
Let $dA$ denote the area measure on the unit disk. The Bergman space $A_2=A_2(\DD)$ of the unit disk consists of all 
functions analytic in $\DD$ for which 
$$ \|  f \|_{A_2}=\left(\int_\DD |f(z)|^2 dA(z)\right)^{1/2} <\infty,$$
The quantity $\|\, \cdot \, \|_{A_2}$ is called the norm of the function.  It is a reproducing kernel Hilbert space, where
$$
k^B_w(z)=\displaystyle{\frac{1}{(1-\bar{w}z)^2}},\, \, \, \, z,w \in \DD,
$$
are the reproducing kernels. Note that the argument of the above proposition, based on degrees, cannot be carried over with the Bergman kernel. We shall establish that  condition (\ref{det}) does not hold in general for $A_2$ by means of a counterexample. It does hold though in several cases. Let us establish the case $n=2$ in the following remark (the case of singletons $\aa=\{a\}$, $\bb=\{b\}$ will be treated in the next section).
\begin{rem}
The kernel function $k^B_w(z)$ is conformally invariant \cite[Chap. I]{ds}: if $\varphi:\mathbb{D}\to\mathbb{D}$ is a conformal map, then  $$
k_{w}^B(z)=k^B_{\varphi(w)}(\varphi(z)) \varphi'(z)\overline{\varphi'(w)}.
$$
In particular, it follows that $\det(k^B_{b_j}(a_i)_{i,j})\ne 0$ if and only if  $\det(k^B_{\varphi(b_j)}(\varphi(a_i))_{i,j})\ne 0$. 
 Let  $\aa=\{a_1,a_2\}$ and $\bb=\{b_1,b_2\}$  be disjoint subsets of $\DD$. Then to compute
$$
\det \left( \begin{array}{cc} k_{b_1}^B(a_1) & k_{b_2}^B(a_1) \\ k_{b_1}^B(a_2) & k_{b_2}^B(a_2)\end{array} \right),
$$
 we can take the automorphism of the disk $\varphi(z)=\frac{z-a_1}{1-\overline{a_1}z}$, and assume that $a_1=0$. Note that
$$
\det\begin{pmatrix}  1 &  1  \\ \frac{1}{(1-\overline{b_1}a_2)^2} &   \frac{1}{(1-\overline{b_2}a_2)^2} \end{pmatrix}\neq 0,
$$
holds exactly when $a_2\neq a_1 (=0)$ and $b_1\neq b_2$. .
\end{rem}

The next result, shows that if $\aa, \bb$ are close to the origin, then condition (\ref{det}) holds for $n\geq 3$.
\begin{teo}\label{when is true for the Bergman}
	Let  $\delta=0.195$, and $\aa=\{a_1,\dots, a_n\}$ , $\bb=\{ b_1,\dots, b_n\}$ be subsets of  $B\left(0,\sqrt{\frac{\delta}{1+\delta}}\right)\subset \DD$, with $\aa\cap\bb=\emptyset$.   Then 
	\begin{equation}
	\label{eq: det mat nucleo bergman no nulo}
	\det \left(\frac1{\left(1-a_i \overline{b}_j\right)^2}\right)_{i,j}\neq 0.
	\end{equation}
\end{teo}
\begin{proof}
	Using Borchardt's identity (see Corollary 5.1 in \cite{agj}) we can write
	\begin{equation}
	\label{eq: det y per}
	\det \left(\frac{1} {(1-a_i\overline{b}_j)^2}\right)_{i,j} = \det\left(\frac{1}{ 1-a_i\overline{b}_j }\right)_{i,j} \text{per}\left(\frac{1}{ 1-a_i\overline{b}_j }\right)_{i,j}
	\end{equation}
	where $\text{per}(m)$ denotes the permanent of the matrix $m$. Suppose first 
	that $a_i\neq 0$ for $i=1,\dots,n$, then
	$$
	\det\left(\frac{1}{ 1-a_i\overline{b}_j  }\right)_{i,j}=\left(\prod_{i=1}^n 1/a_i\right)\det \left(\frac{1}{ (1/a_i)-\overline{b}_j }\right)_{i,j}
	$$
	Since  $|1/a_i|>1>|\overline{b}_j|$ then $\left(\frac{1}{(1/a_i)-\overline{b}_j}\right)_{i,j}$ is a Cauchy matrix and its determinant has a known closed form. Moreover, the assumptions $a_r\ne a_s$, $b_r\ne b_s$ for $r\ne s$, imply that this determinant is non zero:
	\begin{equation}
	\label{eq: det}
	\det \left(\frac{1}{(1/a_i)-\overline{b}_j}\right)_{i,j}={{\prod_{r>s} (1/a_r-1/a_s)(\overline{b}_r-\overline{b}_s)}\over {\prod_{i=1}^n \prod_{j=1}^n (1/a_i-\overline{b}_j)}}\neq 0.  
	\end{equation}
	In case $a_{i_0}=0$ for some $i_0$ then  $b_j\neq 0$ must hold for all $j=1,\dots,n$ because $a_i\neq b_j$ for all $i,j$. Then we can reason similarly with the matrix $\left(\frac{1}{(1/\overline{b}_j)-a_i}\right)_{i,j}$.
	Therefore it can be proved that $ \det\left(\frac{1}{ 1-a_i\overline{b}_j }\right)_{i,j} \neq 0$ in any case.
	
	Observe now that the condition $a_i, b_j\in  	B\left(0,\sqrt{\frac{\delta}{1+\delta}}\right)$ implies that $|a_i \overline{b}_j|<\left(\sqrt{\frac{\delta}{1+\delta}}\right)^2=\frac{\delta}{1+\delta} $.
	Then using that $|a_{i} \overline{b}_j| < \frac{\delta}{1+\delta}<1$ we can state that
	\begin{equation*}
	\label{eq: condic z/1-z}
	\left|\frac{1}{1-a_{i} \overline{b}_j}-1\right|=\frac{|a_{i} \overline{b}_j|}{ |1-a_{i} \overline{b}_j|}   <  \frac{|a_{i} \overline{b}_j|}{ 1-|a_{i} \overline{b}_j|}  <   \frac{ \delta/(1+\delta) }{ 1 - \delta/(1+\delta) }  = \delta = 0.195.
	\end{equation*}
	Then applying Theorem 1.2 of \cite{barv} follows that
	$
	\text{per}\left(\frac{1}{1-a_i \overline{b}_j}\right)_{i,j}\neq 0
	$
	and considering \eqref{eq: det} we obtain that \eqref{eq: det y per} is non zero, which completes the proof.	 
\end{proof}
\begin{rem}
	Note that in the previous theorem the assumption  $a_i, b_j\in B\left(0,\sqrt{\frac{\delta}{1+\delta}}\right)$
	can be weakened to $|a_{i} {b}_j| < \frac{\delta}{1+\delta}$ and $a_i, b_j\in\DD$.
\end{rem}
\begin{coro}
	Let  $\delta=0.195$,   $\aa=\{a_1,\dots, a_n\}$ , $\bb=\{ b_1,\dots, b_n\}\subset \DD$, with $\aa\cap\bb=\emptyset$ and $|a_i|<\frac{\delta}{1+\delta}$ for $i=1,\dots, n$.   Then $
	\label{eq: det mat nucleo bergman no nulo}
	\det \left(\frac1{\left(1-a_i \overline{b_j}\right)^2}\right)_{i,j}\neq 0.
	$
\end{coro}
Condition (\ref{det}) does not hold in general in the Bergman space. Consider the following example, in the case $n=3$:
\begin{ejem}	
	Let $b_1=-\frac{257}{367}-\frac{17 }{45}i,\ b_2=-\frac{62}{311}+\frac{337 }{376}i,\ b_3=\frac{356}{403}+\frac{86 }{403}i$ and $c_1=\frac{33}{68}-\frac{19}{411}i,\ c_2=\frac{244}{353}-\frac{16
	}{343}i, \ c_3=\frac{43}{85}-\frac{254 }{335}i$ (with $b_1, b_2, b_3\in\DD$), then the function 
	$$
	f(z)=\sum_{j=1}^3\frac{c_j}{(1-z\bar{b_j})^2}
	$$
	has three zeros inside $\DD$ that are approximately
	$
	z_1= -0.837508  + 0.3451006  \ i 
	$, 
	$
	z_2=		0.1723709  - 0.832953  \ i 
	$
	and
	$
	z_3=		0.466866  + 0.855772 \ i.
	$ 
	Therefore the columns of the matrix $\left( \frac{1}{(1-z_i\overline{b_j} )^2} \right)_{i,j=1}^3$ are linearly dependent and  its determinant vanishes.
	
	A different process to obtain exact complex numbers $a_i$, $b_j\in \DD$, $i,j=1,2, 3$ such that $\det\left( \frac{1}{(1-a_i\overline{b_j} )^2} \right)_{i,j=1}^3=0$ is the following. With the notation of the previous example we may approximate $z_1$ with $a_1=-\frac{67}{80} + \frac{88 i}{255}$
	and $z_2$ with $a_2=\frac{101}{586}-\frac{369 i}{443}$, $a_1, a_2\in\DD$. Then consider $a_3=z$ and define the polynomial $p(z)$ consisting of the numerator of the map $\frac{1}{(z-a_1)(z-a_2)}\det\left( \frac{1}{(1-a_i\overline{b_j} )^2} \right)_{i,j=1}^3$. Since $z=a_1$ and $z=a_2$ are roots of $p$ it can be proved that the $p(z)$ has degree $2$. Then two roots of $p$ can be found explicitly. One of them belongs to $\DD$. Then if we chose this root as $a_3$ we obtain $b_1, b_2, b_3$ and $a_1, a_2, a_3$ in the disk $\DD$ such that $\det\left( \frac{1}{(1-a_i\overline{b_j} )^2} \right)_{i,j=1}^3=0$.
\end{ejem}
We have also found numerical examples for $n=4, \dots, 8$ such that 
$\det\left( \frac{1}{(1-a_i\overline{b_j} )^2} \right)_{i,j=1}^n=0$.

\medskip

\noindent \textbf{Segal-Bargmann space.} Let $H(\CC)$ be the space of holomorphic functions in $\CC$. Our next example is given by the Bargmann space, or Segal-Bargmann space, which is defined as
$$
\mathcal{F}^1=\left\{  \,  f \in H(\CC) \, : \,  \int_\CC |f(z)|^2 e^{-|z|^2} dz < \infty  \, \right\}.
$$  
It has a reproducing kernel, $k^S_w(z)=e^{z\bar{w}}$, $z,w \in \CC$. Now condition  (\ref{det})  does not always hold in this space. 
An easy necessary and sufficient condition can be established for the case of two zeros. Take the sets $\aa=\{a_1,\, a_2\} \subseteq \CC$, and  $\bb=\{b_1,\, b_2\} \subseteq \CC$ with $\aa \cap \bb=\emptyset$. There is a unique geodesic joining $\z_\aa$ and $\z_\bb$ if and only if
$$
\det \left( \begin{array}{cc} e^{a_1\bar{b}_1} & e^{a_1\bar{b}_2} \\ e^{a_2\bar{b}_1} & e^{a_2\bar{b}_2}\end{array} \right)=e^{a_1\bar{b}_1 + a_2\bar{b}_2} - e^{a_1\bar{b}_2+ a_2\bar{b}_1} \ne 0,
$$
which can be rewritten as
$$
e^{a_1 \bar{b}_1 + a_2 \bar{b}_2-(a_1 \bar{b}_2 + a_2 \bar{b}_1)} \neq 1.
$$
This, in turn, is equivalent to the condition:
$$
(a_1-a_2)(\bar{b}_1-\bar{b}_2)\neq 2 k \pi i, \, \, \, k \in \ZZ.
$$
When this last condition does not hold, there are infinitely many geodesics in $\mathcal{F}^1$ joining $\z_\aa$ and $\z_\bb$.

%\begin{ejem}
%Let $H(\CC)$ be the space of holomorphic functions in $\CC$. Our next example is given by the Bargmann space, or Segal-Bargmann space, which is defined as
%$$
%\mathcal{F}^1=\left\{  \,  f \in H(\CC) \, : \,  \int_\CC |f(z)|^2 e^{-|z|^2} dz < \infty  \, \right\}.
%$$  
%It has a reproducing kernel, $k^S_w(z)=e^{z\bar{w}}$, $z,w \in \CC$. Now condition  (\ref{det})  does not hold in the Bargmann space. 
%An easy necessary and sufficient condition can be established for the case of two zeros. Take the sets $\aa=\{a_1,\, a_2\} \subseteq \CC$, and  $\bb=\{b_1,\, b_2\} \subseteq \CC$. There is a geodesic joining $\z_\aa$ and $\z_\bb$ if and only if
%$$
%\det \left( \begin{array}{cc} e^{a_1\bar{b}_1} & e^{a_1\bar{b}_2} \\ e^{a_2\bar{b}_1} & e^{a_2\bar{b}_2}\end{array} \right)=e^{a_1\bar{b}_1 + a_2\bar{b}_2} - e^{a_1\bar{b}_2+ a_2\bar{b}_1} \ne 0,
%$$
%which can be rewritten as
%$$
%e^{a_1 \bar{b}_1 + a_2 \bar{b}_2-(a_1 \bar{b}_2 + a_2 \bar{b}_1)} \neq 1.
%$$
%This, in turn, is equivalent to the condition:
%$$
%(a_1-a_2)(\bar{b}_1-\bar{b}_2)\neq 2 k \pi i, \, \, \, k \in \ZZ.
%$$
%\end{ejem}

\medskip

%\noindent \textbf{Dirichlet space.}{\bf ACA VIENE EL EJEMPLO DEL ESPACIO DE DIRICHLET}

\section{Singletons}\label{singletons}

Let $\h$ be a reproducing kernel Hilbert space consisting of functions on a set $X$. 
%\begin{hyp}\label{hip cond 1 and 2}
Throughout this section, we assume that the reproducing kernel satisfies the following conditions:
\begin{itemize}
\item[i)] $k_a$ is not the zero function, for any $a \in X$;
%\item[i)] $\z_{a}\cap\z_b^\perp=\z_a^\perp\cap\z_b=\{0\}$ for every $a,b \in X$;
\item[ii)] the set $\{ k_a \, , k_b \}$ is linearly independent  if $a\neq b$.
\end{itemize}
%\end{hyp}

In our next remark we recall three important metrics in the context of reproducing kernel Hilbert spaces.
We follow the exposition in  \cite{ARSW11}. We omit  the proofs, the details can be found in this work and the references therein.

\begin{rem}\label{3 metrics}
By the first condition above, we can normalize the functions $k_a$. On the other hand, the three metrics below turn into 
pseudo-metrics if and only if the second condition is not assumed. This follows  straightforward using the equality case in the Cauchy-Schwartz inequality.  

\medskip

\noindent 1.  %For $a \in X$,  we write $\ell_a=\frac{k_a}{\|k_a\|}$, the normalized kernel function.
The first metric is given by
$$
%\delta(a,b)=\delta_\h(a,b):=\sqrt{1- |\PI{\ell_a}{\ell_b}|^2}, \, \, \, \, \, a,b \in X.
%\delta(a,b)=\delta_\h(a,b):=\sqrt{1- \left|\PI{\frac{k_a}{\|k_a\|}}{\frac{k_b}{\|k_b\|}}\right|^2}, \, \, \, \, \, a,b \in X.
\delta(a,b)=\delta_\h(a,b):=\sqrt{1- \left(\frac{|\PI{k_a}{k_b}|}{\|k_a\|\|k_b\|}\right)^2}.
%{\|k_a\|}}{\frac{k_b}{\|k_b\|}}\right|^2}, \, \, \, \, \, a,b \in X.
$$
See \cite{AMc02} for a proof of the triangle inequality. It can be interpreted as a measure between points in $X$, which takes into account properties of $\h$. For instance, 
the following relation holds:
$$
\delta(a,b)=\frac{\sup\{ |f(b)| : f \in \h, \, \|f\|=1, \, f(a)=0 \}  }{\sup\{ |f(b)| : f \in \h, \, \|f\|=1 \}  }.
$$

Let $P_a$ be the orthogonal projection onto the subspace generated by $k_a$. The metric $\delta$ is also useful to provide Lipschitz estimates of the Berenzin transform. For this purpose in \cite{C07}, the following characterization of $\delta$ was proved:
\begin{equation}\label{eq with norm}
\delta(a,b)=\|P_a - P_b \|=2^{-1/p}\|P_a - P_b \|_p, \, \, \, p\geq 1 .
\end{equation}
Here $\| \,\, \|_p$ denote the Schatten $p$-norms. Indeed, this was first showed for the operator and trace norms \cite{C07}, and later observed for the other $p$-norms \cite{ARSW11}.
In addition, $\delta$ might be viewed as a generalization of the pseudo-hyperbolic metric for arbitrary reproducing kernel Hilbert spaces. Recall that the pseudo-hyperbolic metric is defined by 
$$
\rho(a,b)=\left| \frac{a-b}{1-\bar{a}b}	\right|, \, \, \, a, b \in \DD.
$$
It is well-known that for the Hardy space $H^2$ we have
$$
\delta_{H^2}(a,b)=\rho(a,b).
$$		
\noindent 2. The Skwarcy\'nski metric was first considered  for the Bergman kernel on a domain \cite{MPS85}. It can be extended to an arbitrary reproducing kernel Hilbert space as
$$
%\hat{\delta}(a,b)=\hat{\delta}_\h(a,b):=\sqrt{1- \left|\PI{\frac{k_a}{\|k_a\|}}{\frac{k_b}{\|k_b\|}}\right|}, \, \, \, a,b \in X.
\hat{\delta}(a,b)=\hat{\delta}_\h(a,b):=\sqrt{1- \frac{|\PI{k_a}{k_b}|}{\|k_a\|\|k_b\|}}, \, \, \, a,b \in X.
$$
 It can be realized  as a multiple of a  quotient metric as follows. Let $P(\h)$ the projective space over $\h$, i.e. $P(\h)\simeq S(\h)/ \sim$ , where $S(\h)$ denotes the unit sphere of $\h$ and $f\sim g$ if $f=\lambda g$, for some $\lambda \in \TT$. Then the Skwarcy\'nski distance between $a$ and $b$ can be computed as a multiple of the quotient distance between the classes $\left[\frac{k_a}{\|k_a\|}\right]$ and $\left[\frac{k_b}{\|k_b\|}\right]$,  that is,
$$
\hat{\delta}(a,b)=\frac{1}{\sqrt{2}}\inf \left\{  \, \left\| \frac{k_a}{\|k_a\|} - \lambda \frac{k_b}{\|k_b\|}  \right\| \, : \, \lambda \in \TT  \, \right\}.
$$ 

\noindent 3. Our third metric is usually known as the Kobayashi metric \cite{K59}.  In order to define it, recall that the tangent space $(TP(\h))_{[f]}$ of the projective space at $[f]$ is $(TP(\h))_{[f]}=(TS(\h))_{f}/\sim$, where $v \sim  w$ if $v-w=i a f$, $a \in \R$ and  the tangent space to sphere is given by $(TS(\h))_{f}=\{ v \in \h : \Re\PI{v}{f}=0 \}$. The following Riemannian metric is the infinite dimensional version of the Fubini-Study metric: for $[v] \in (TP(\h))_{[f]}$, 
$$
\| [v] \|_{[f]}=\mathrm{dist}(v,i\R f)=(\|v\|^2    -  |\PI{v}{f}|^2)^{1/2}.
$$
Given a piecewise smooth curve $\gamma:[0,1]\to P(\h)$, its  length  is then measured by
$$
L(\gamma)=\int_0^1 \left( \left\|\dot{\Gamma}(t)\right\|^2 - \left|\PI{\dot{\Gamma}(t)}{\Gamma(t)}\right|^2\right)^{1/2} dt.
$$
Here $\Gamma \subseteq  S(\h)$ is any piecewise smooth lift of $\gamma$, i.e. $[\Gamma(t)]=\gamma(t)$ for all $t \in [0,1]$.
The Kobayashi metric is  defined by using the corresponding geodesic distance:
$$
\check{\delta}(a,b):=\inf \left\{ \, L(\gamma) \,  : \, \gamma \text{ piecewise smooth joining }  \left[\frac{k_a}{\|k_a\|}\right] \text{ and } \left[\frac{k_b}{\|k_b\|}\right]\,  \right\}.
$$
\end{rem}

\medskip

Let us consider the  case when  the sets $\aa$ and $\bb$ consist of single terms $a$ and $b$ in $X$. We denote 
by $\z_a$ and $\z_b$ the corresponding subspaces of functions in $\h$ that vanish in the points $a$ and $b$, respectively.   
We have defined the length of curves by using a Finsler norm given by the operator norm. 
Recall that if $\gamma:[0,1] \to \gr$ is a piecewise smooth curve, then its length is measured by
$$
L(\gamma)=\int_0^1 \| \dot{\gamma}(t)\| dt.
$$
Thus, we have a geodesic distance $d(\s,\t)$ defined as 
the infimum of all the length of piecewise smooth curves in $\gr$ joining the subspaces $\s$ and $\t$. In particular, this allows us to introduce 
another metric in $X$:
$$
\Gamma(a,b)=\Gamma_\h(a,b):=d(\z_a, \z_b), \, \, \, a, b \in X.
$$
The relation of this metric with the three metrics of the previous remark is as follows. Notice that the second item gives another
proof in geometric terms of the above relation (\ref{eq with norm}).

\begin{teo}\label{coincidencia} 
Let $\h$ be a reproducing kernel Hilbert space of function on a set $X$. The following assertions hold: 
\begin{enumerate}
\item[1.] $\check{\delta}=\Gamma$;
\item[2.] $\delta(a,b)=\sin(\Gamma(a,b))=\|  P_{\z_\aa} - P_{\z_\bb}  \|=2^{-1/p}\|P_{\z_\aa} - P_{\z_\bb} \|_p$, $p\geq 1$;
\item[3.] $ \hat{\delta}(a,b)=\sqrt2 \ \sin\left(\frac12 \Gamma(a,b)\right).$
\end{enumerate}
\end{teo}
\begin{proof}
We compute $\Gamma$ first in the case where $\z_a\cap \z_b^\perp=\z_a^\perp \cap \z_b=\{ 0 \}$. Then $\h_0$, the generic part where $P_{\z_a}$ and $P_{\z_b}$ act non-trivially, is a $2$-dimensional space generated by $k_a$ and $k_b$. Let us compute $X_{a,b}$ the exponent of the unique minimal geodesic joining these subspaces. Clearly $X_{a,b}=0$ in $\h_{11}$. Consider the orthonormal basis 
$\{e_1,e_2\}$ of $\h_0$, given by $e_1=\frac{1}{\|k_a\|}k_a$ and $e_2=\frac{1}{\|h\|}h$,
where
$$
h=k_b-\frac{\langle k_b,k_a \rangle}{\|k_a\|^2} k_a.
$$
Then elementary computations show that in this basis, 
$$
P_{\z_{a}}|_{\h_0}=\left( \begin{array}{cc} 0 & 0 \\ 0 & 1 \end{array}\right) \ \hbox{ and } \ 
P_{\z_{b}}|_{\h_0}=\left( \begin{array}{cc} 1-|\gamma|^2 & (1-|\gamma|^2)^{1/2}\gamma \\ 
(1-|\gamma|^2)^{1/2}\bar{\gamma} & |\gamma|^2 \end{array}\right)
$$
where 
$$
\gamma=\frac{\langle k_a,k_b \rangle}{\|k_a\|\|k_b\|}.
$$
Let $\gamma=e^{i\theta}\cos(x)$, where $\cos(x)=|\gamma|$ for $0<x<\pi/2$, and consider the unitary matrix
$$
U_\theta=\left( \begin{array}{cc} e^{i\theta/2} & 0 \\ 0 & e^{-i\theta/2} \end{array} \right).
$$
Then 
$$
P_{\z_{a}}|_{\h_0}=U_\theta P_{\z_{a}}|_{\h_0} U_{-\theta} \ \hbox{ and } \ P_{\z_{b}}|_{\h_0}=
U_\theta \left( \begin{array}{cc} \cos(x)^2 & \cos(x)\sin(x) \\ \cos(x)\sin(x) & \sin(x)^2 \end{array} \right) U_{-\theta}.
$$
Therefore, the  co-diagonal, selfadjoint matrix (of norm less than $\pi/2$) which is the exponent of the unique geodesic joining these projections is
$$
X_{a,b}=U_\theta \left( \begin{array}{cc} 0 & -i x \\ ix & 0 \end{array} \right) U_{-\theta}= \left( \begin{array}{cc} 0 & -i x e^{i\theta} \\ i x e^{-i\theta} & 0 \end{array} \right) .
$$
Thus, the geodesic distance between $\z_a$ and $\z_b$ is 
$$
d(\z_a,\z_b)=\|X_{a,b}\|=x=\arccos \left(\frac{|\langle k_a,k_b \rangle|}{\|k_a\|\|k_b\|}\right).
$$
In the case where $\z_a \cap \z_b^\perp\neq \{ 0\}$ or  $\z_b \cap \z_a^\perp\neq \{ 0\}$, note that $\z_b^\perp \subseteq \z_a $ and 
$\z_a^\perp \subseteq \z_b$, or equivalently $k_a(b)=\PI{k_a}{k_b}=0$. Then $\dim \z_a \cap \z_b^\perp= \dim \z_b \cap \z_a^\perp=1$, so there are infinitely many geodesics joining
 $\z_a$ and $\z_b$, and the geodesic distance equals $\pi /2$. In both cases, we thus  obtain 
\begin{equation}\label{expres of Gamma}
\Gamma(a,b)=\arccos \left(\frac{|\langle k_a,k_b \rangle|}{\|k_a\|\|k_b\|}\right), \, \, \, \, a,b \in X.
\end{equation}
In order to show that $\check{\delta}=\Gamma$, we observe that the Fubini-Study metric is invariant under the action of the unitary group.
Also identyfing the Hilbert space $\h\simeq \ell^2$, we may assume that 
$$
\frac{k_a}{\|k_a\|}=(1,0,0, \ldots );  \, \, \, \, \, \frac{k_b}{\|k_b\|}=\left( \frac{1}{(c^2 + 1)^{1/2}}, \frac{c}{(c^2 + 1)^{1/2}} ,0, \ldots \right), \, \, \, 0\leq c \leq \infty.
$$
This fact was observed in \cite{K59}, where it is also computed that $\check{\delta}(a,b)=\arctan(c)$. Our Finsler metric is also unitarily invariant, thus we can also assume that the above unit vectors have that form, so that $\Gamma(a,b)=\arccos(c^2+1)^{-1/2}$. This gives $\check{\delta}=\Gamma$.

From the above expression of $\Gamma$ given in (\ref{expres of Gamma}), it follows immediately that $\delta(a,b)=\sin(\Gamma(a,b))$. 
It is known that (see for instance the survey article \cite{timisoara}), for any given pair of projections that can be joined by a geodesic curve of the Grassmann manifold, it holds that
$$
\sin(d(P,Q))=\|P-Q\|. 
$$
This also can be generalized to $p$-norms by similar arguments.
   Finally notice that $\hat{\delta}(a,b)= (1- \cos(\Gamma(a,b)))^{1/2}$, which implies that $ \hat{\delta}(a,b)=\sqrt2 \ \sin\left(\frac12 \Gamma(a,b)\right)$.
\end{proof}
%OBSERVACIONES: 1- $\delta$ es una m�trica en muchos RKHS porque $\k_x \neq 0$ en general - 2 $\hat{\delta}$ tambi�n necesita $\k_x \neq 0$ 
%y chequear condici�on de LI que puse - 3- Estamos estudiando la completitud de $X$ con $\Gamma$ al ver si $Gr_p$ cerrado - Al cambiar el domino
%el n�cleo de Bergman es otro...esto da muchas clases de ejemplos - 4- Breneman ya se le ocurri� lo de tender a infinto la norma del nucleo en el borde
In the previous proof we have seen that
\begin{equation*}
\Gamma(a,b)=\| X_{a,b}\|=\arccos \left(\frac{|\langle k_a,k_b \rangle|}{\|k_a\|\|k_b\|}\right), \, \, \, \, a,b \in X.
\end{equation*} 
\begin{ejems}\label{Gamma examples}
For instance $\Gamma$ can be computed in the following spaces:
\begin{enumerate}
\item[1.] Hardy space: $\Gamma_{H^2}(a,b)=\arcsin\left(\left|\frac{a-b}{1-\bar{a}b}\right|\right)=\arcsin(\rho(a,b))$, $a,b \in \DD$
\item[2.] Bergman space: $\Gamma_{A_2}(a,b)=2 \arcsin\left(2^{-1/2}\left|\frac{a-b}{1-\bar{a}b}\right|\right)=2\arcsin(2^{-1/2}\rho(a,b))$, $a,b \in \DD$.
\item[3.] Bargmann-Segal space: $\Gamma_{\mathcal{F}^1}(a,b)=\arccos(e^{-\frac{1}{2}|a-b|^2})$,  $a,b \in \CC$.
\end{enumerate}
\end{ejems}

Points $a\in X$ can be regarded as subspaces $\z_a\in \gr$. Clearly, the map $a\mapsto\z_a$ is one to one. This follows from our assumption at the beginning of this section that the set $\{ k_a \, , k_b \}$ is linearly independent  if $a\neq b$. Furthermore, the map
$$
(X,\Gamma) \ni a \mapsto P_{\z_a}\in \grp:=\{\z_a: a\in X \}
$$
is a homeomorphism. This follows from the relations in Theorem \ref{coincidencia}. Note that we may also endow $X$ with any other of the metrics
$\delta$, $\hat{\delta}$ or $\check{\delta}$, they all give the same topology on $X$. 

We now investigate when 
$
\grp=\{\z_a: a\in X \}
$ 
is a closed subset of $\gr$. Since  $\gr$ is closed in $\b(\h)$, $\grp$ is closed in $\gr$ if and only if 
$\grp$ is closed in $\b(\h)$. More interesting, we have that $\grp$ is closed in $\gr$ if and only if 
the metric space $(X, \Gamma)$ is complete. This again follows immediately from the relation  between $\Gamma$ and the operator norm established
in Theorem \ref{coincidencia}.

For our purpose, we shall need that the space $\h$ has the following property.
\begin{hyp}\label{hipo}
Assume that $X$ is a subset of  $Y=\R^n$ or $Y=\CC^n$. Denote by $\hat{Y}=Y \cup \{ \infty\}$ the one-point compactification of $Y$.
Let $\partial X$ be the boundary of $X$ in $\hat{Y}$. We suppose that the following hold:  if $(w_n)$ is a sequence  in $X$ such that 
$w_n \to w \in \partial X\setminus X$, then
\[ 
\lim_{n \to \infty} \frac{k_{w_n}(z)}{\|k_{w_n}\|}=0, \, \, \, \forall \, z \in X.
\]
\end{hyp}

\noindent We now give a useful  sufficient condition and examples regarding this hypothesis.

\begin{rem}\label{suff cond}
\noindent Clearly, the above hypothesis  is satisfied if the following hold:
\begin{enumerate}
\item[1.] If $(w_n)$ is a sequence  in $X$ such that 
$w_n \to w \in \partial X \setminus X$, then $\|k_{w_n}\| \to \infty$. 
\item[2.] For every $z \in X$, there exists a constant $C_z>0$ such that $|k_w(z)|\leq C_z$ for all $w \in X$. 
\end{enumerate}
In the case of $X$ being a bounded domain of $\CC^n$, condition 1. dates back to the work of Brenermann \cite{B55}. Under this assumption $X$ turns out to be complete with respect to the Bergmann metric $\delta_{B}$ defined by means of the Bergman kernel. Although we shall not present the Bergmann metric here, we observe that
$\delta_{B}=2 \check{\delta}$ (see \cite{K59, MPS85}, or more generally, \cite[Prop. 9]{ARSW11}), and then by Theorem \ref{coincidencia}, $\delta_{B}=2\Gamma$. Thus, condition 1. implies that $(X,\Gamma)$ is a complete metric space, or equivalently $\grp$ is closed in $\gr$, when $X$ is a bounded domain of $\CC^n$ and $\h$ is the Bergmann space associated to it.    
\end{rem}

\begin{ejems}
$1.$ The Hardy and Bergman spaces of the disk clearly satisfy both conditions in Remark \ref{suff cond}. 
A generalization of the Hardy space, which also satisfies  these conditions, is the Drury-Arveson  space $D_n$. It  consists in all the holomorphic functions on $\mathbb{B}_n$, the unit ball of $\CC^n$, $n\geq 1$, equipped with the reproducing kernel 
$$
k^{D_n}(z,w)=\frac{1}{1-\sum_{j=1}^n z_j \bar{w}_j }\, , 
$$ 
where $z=(z_1, \ldots, z_n)\in \mathbb{B}_n$, $w=(w_1, \ldots, w_n) \in \mathbb{B}_n$. 
%This example can further be generalized to infinite dimensions by taking $\ell^2(\mathbb{N})$ instead of $\CC^n$, and its corresponding inner product in the definition of the kernel.  

\medskip

\noindent $2.$ Notice that conditions 1. and 2. are independent of each other. There are examples in which only one of them holds. 
Given $\beta=(\beta_n)$ a sequence of positive real numbers, set $
R=\lim \inf \beta_n^{-1/n}$. The weighted Hardy space $H^2_\beta$ consists of  all the analytic functions in the disk $B_R(0)$ such that 
$$
 \|f\|_\beta:=\sum_{n=0}^\infty \beta_n^2|a_n|^2< \infty, 
$$
whenever $f(z)=\sum_{n=0}^\infty  a_n z^n$ on $B_R(0)$. 
For details we refer to \cite{paulsen}. Special choices of the weights $(\beta_n)$ give the following 
spaces of analytic functions treated in this work: $\beta_n=1$ (usual Hardy space), $\beta_n=\frac{1}{\sqrt{n+1}}$ (Bergman space) and $\beta_n=\sqrt{n!}$ (Bargmann-Segal space). The reproducing kernel of $H_\beta^2$ is given by
$$
k_w(z)=\sum_{n=0}^\infty \frac{\bar{w}^nz^n}{\beta_n^2}, \, \, \, \, \, \, \, w,z \in B_R(0).
$$
Now we can show that condition 1. does not necessary hold true. Take $\beta_n=n+1$, then $R=1$, and for $a \in \DD$,
\begin{equation}\label{pi 6}
\|k_a\|^2=k_a(a)=\sum_{n=0}^\infty \left(\frac{|a|}{n+1}\right)^2 \leq \sum_{n=0}^\infty \frac{1}{(n+1)^2} =\frac{\pi^2}{6}.
\end{equation}
This implies that $\|k_a\| \not\to \infty$ as $|a| \to 1$. However, estimating in a similar way, we can find that condition 2. holds in 
this space. 

On the other hand, the Bargmann-Segal space satisfies condition 1. since $\|k_w\|^2=e^{|w|^2}\to \infty$ as $w \to \infty$. But it does not satisfy condition 1. For fix $z \in \CC$, $|z|>1$, and take the sequence $w_n=z|z|^n$, then $k_{w_n}(z)=e^{|z|^{n+2}}\to \infty$ as $n \to \infty$.   

\medskip

\noindent 3. Notice that the above Hardy space with  weight $\beta_n=n+1$ provides an example in which Hypothesis \ref{hipo} does not hold true. Take $z=0$ and $(w_n)$ a sequence in the unit disk such that $w_n \to 1$. Then using the estimate in (\ref{pi 6}), we find that 
$$
\frac{k_{w_n}(0)}{\|k_{w_n}\|}=\frac{1}{\|k_{w_n}\|}\geq \frac{\sqrt{6}}{\pi}.
$$
Finally we observe that Hypothesis \ref{hipo} is more general than the conditions in Remark \ref{suff cond}. For instance, in the Segal-Bargmann space
we have that Hypothesis \ref{hipo} holds true:  
$\lim_{w \to \infty} e^{z\bar{w}-|w|^2}=0$ for every $z \in \CC$. 
Another example where condition 1. does not hold, but Hypothesis \ref{hipo} can be verified is the Sobolev space in $\R$, which consists of absolutely continuous functions $f$ such that $f, f'  \in L^2(\R)$. It has the inner product
$$
\PI{f}{g}=\int_\R f(x)\overline{g(x)}dx + \int_\R f'(x)\overline{g'(x)}dx;
$$ 
and the reproducing kernel $k(z,w)=e^{-|z-w|}$, $z,w \in \R$.
\end{ejems}

%For instance, take $\theta$ an infinite Blaschke product (see definition in PAGINA). Then the shift-invariant subspace
%$\h=\theta H^2$ satisfies condition 2., but it does not satisfy condition 1. Indeed, there is a sequence $(w_n)$ consisting of zeros
%of $\theta$ which converges to $w \in  \TT$. Thus, $\|k^{\theta}_{w_n}\|$ 

\begin{prop}\label{closed}
Let $\h$ is a reproducing kernel Hilbert space of functions on a set $X$ which satisfies Hypothesis \ref{hipo}. Assume that 
the kernel $k:X\times X \to \CC$ is a continuous function. Then $\grp$ is closed in $\gr$.
\end{prop}
\begin{proof}
At the beginning  of this section we assume that $k_a$ is not the zero function for any $a \in X$. Denote by $\ell_a$ the normalized kernels, $\ell_a=\frac{1}{\|k_a\|} k_a$. Then the orthogonal projection $P_{\z_a}$ onto $\z_a$ is given by $P_{\z_a}=1-\ell_a\otimes\ell_a$, where as is usual notation, $f\otimes g$ denotes the rank one operator $h\mapsto \langle h, g\rangle f$ ($f,g,h\in\h$). Let $a_n\in X$ and suppose that $P_{\z_{a_n}}\to P$ in the norm topology of $\b(\h)$ (which is the topology in $\gr$). We must show that $P=1-\ell_b\otimes \ell_b$ for some $b\in\DD$. First note that the projections $\ell_{a_n}\otimes\ell_{a_n}$ converge to a rank one projection. Indeed, let $f\in R(1-P)$ with $\|f\|=1$, and suppose there exists $g\in R(1-P)$ such that $g\perp f$. Since $\langle f,\ell_{a_n}\rangle \ell_{a_n}=(\ell_{a_n}\otimes\ell_{a_n})(f)\to f$, it follows, on one hand, that
$$
|\langle f,\ell_{a_n}\rangle|^2=\langle \ell_{a_n}\otimes\ell_{a_n}(f),f\rangle\to \langle f,f\rangle =1,
$$
i.e. $|\langle f,\ell_{a_n}\rangle|\to 1$. On the other hand,
$$
\langle f,\ell_{a_n}\rangle \langle \ell_{a_n}, g\rangle \to \langle f,g\rangle,
$$
and therefore $\langle g,\ell_{a_n}\rangle\to 0$, and $\ell_{a_n}\otimes\ell_{a_n}(g)\to (1-P)(g)=0$. That is, $\ell_{a_n}\otimes\ell_{a_n}\to f\otimes f$. 
%The sequence $\langle f,\ell_{a_n}\rangle$ is bounded and thus has a convergent subsequence. We may suppose that $\langle f,\ell_{a_n}\rangle\to \alpha$. Note that $|\alpha|=1$. We replace $f$ with $f_0=\bar{\alpha}f$, so that
%$$
%\langle f_0,\ell_{a_n}\rangle \to 1 \hbox{ and } \ell_{a_n}\otimes \ell_{a_n}(f_0)\to f_0.
%$$
%Note also that $f\otimes f=f_0\otimes f_0$.
Now recall that $X$ is contained in $Y=\R^n$ or $Y=\CC^n$. Then  $(a_n)$ has a subsequence (still denoted by $a_n$) which converges to some element  $a \in \hat{Y}$, where $\hat{Y}$ is the one-point compactification of $Y$.   Suppose that $a \in \partial X \setminus X$, then for every $z \in X$, we have
\begin{align*}
|f(z)| = \| (f\otimes f )(k_z)\| = \lim_{n\to \infty} \| (\ell_{a_n} \otimes \ell_{a_n})(k_z)  \| 
   =\lim_{n\to \infty} \frac{|k_{a_n}(z)|}{\| k_{a_n}\|}=0 
\end{align*}
It follows that $f=0$, which is a contradiction. Thus we must have $a \in X$. Using that the kernel is continuous, we obtain that $\ell_{a_n}\to\ell_a$, which implies that $f=\ell_a$.
\end{proof}

\begin{rem}
Notice that Hypothesis \ref{hipo} holds  when $X$ is closed in $\hat{Y}$. On the other hand, we observe that if the kernel is continuous as we have assumed, then all the functions in $\h$ must be continuous (see \cite[Thm. 2.17]{paulsen}). 
\end{rem}

%Theorem \ref{coincidencia} states that in the case where $\h=H^2$, the bijective map 
%$\DD\ni a\mapsto P_{\z_a}\in \grp$ is isometric when $\DD$ is endowed with the metric $\delta$. Let us show that under mild hypothesis on the space $\h$, this map is a homeomorphism from $(X,\Gamma)$ to $\grp$. Of course, by the relations in Theorem \ref{coincidencia}, we can choose 
%for this result any of the other metrics $\delta$, $\hat{\delta}$ or $\check{\delta}$.

%\begin{prop}
%Under the same assumptions of Proposition \ref{closed}, the map 
%$$
%(X,\Gamma) \ni a \mapsto P_{\z_a}\in \grp
%$$
%is a homeomorphism.
%\end{prop}
%\begin{proof}
%Clearly, Theorem \ref{coincidencia} implies that the map $a\mapsto P_{\z_a}$ is continuous. In order to prove that the inverse is continuous take a sequence $(a_n ) \subseteq \DD$ such that $\| P_{\z_{a_n}} - P_{\z_a} \| \to 0$ for some $a \in \DD$. Following the same steps as in the proof of Proposition 3.6 we get that $|a_n|\leq r <1$ for all $n \geq 1$. Then there exists a subsequence $(a_{n_k})$ such that $a_{n_k} \to a_0 \in \DD$.  Thus we find that  $\| P_{\z_{a_{n_k}}} - P_{\z_{a_0}} \| \to 0$, and consequently, $P_{\z_a}=P_{\z_{a_0}}$. Hence $a=a_0$. Now we observe that if we begin the proof with $(a_{n_k})$ in place of $(a_n)$ we obtain that for any sequence $(n_k)$, there exists a subsequence  $(n_{k_j})$ such that $a_{n_{k_j}} \to a$. The latter is equivalent to $a_n \to a$.
%\end{proof}

\medskip
%\bigskip

We end this section discussing the action of Moebius transformations. If $a\in\mathbb{D}$ and $w\in\mathbb{T}$,  denote by $\varphi_{a,w}:\mathbb{D}\to \mathbb{D}$ the Moebius transformation $\varphi_{a,w}(z)=w\displaystyle{\frac{a-z}{1-\bar{a}z}}$. Then, the following is immediate from the expressions of the metric $\Gamma$ for the Hardy and Bergman spaces in Examples \ref{Gamma examples}. 
\begin{coro}
The transformations $\varphi_{a,w}$ are isometries for the metrics $\Gamma_{H^2}$ and $\Gamma_{A_2}$:
$$
\Gamma(\varphi_{a,w}(z_1),\varphi_{a,w}(z_2))=\Gamma(z_1,z_2), \ z_1,z_2\in\mathbb{D},
$$
both for $\Gamma=\Gamma_{H^2}$ and $\Gamma=\Gamma_{A_2}$.
\end{coro}

\begin{rem}
From this result and Theorem \ref{coincidencia}, it follows that for $\h=H^2$ or $\h=A_2$, if $\varphi$ is a Moebius transformation of $\DD$, and $a,b\in\DD$, then
$$
\|P_{\z_{\varphi(a)}}-P_{\z_{\varphi(b)}}\|=\|P_{\z_a}-P_{\z_b}\|.
$$
Denote by $C_\varphi\in\b(\h)$ the composition operator $C_\varphi(f)=f\circ\varphi$. Note that 
$$
\z_{\varphi(a)}=\{f\in\h: C_\varphi f(a)=0\}=C_\varphi^{-1}(\z_a).
$$
In other words, the operators $C_\varphi$, which are invertible but non unitary, preserve the norm distance between (the projections onto) the subspaces of the form $\z_a$, $a\in\DD$.
\end{rem}

\bigskip

Let us consider now the general case of reproducing kernel Hilbert space $\h$ of analytic functions on the disk. Assume that the group $\m(\DD)$ of Moebius transformations acts on $\h$ by means of composition operators, i.e. $f \circ \varphi^{-1} \in \h$ whenever $f \in \h$ and $\varphi \in \m(\DD)$. Then $\m(\DD)$ also acts  on $\grp$, the action is given by
$$
\varphi\cdot\z_a=C_{\varphi^{-1}}(\z_a)=\z_{\varphi(a)}.
$$
%where $C_\psi$ again denotes the composition operator $C_\psi(f)=f\circ \psi$, for a given $\psi:\DD\to\DD$.

\begin{coro}\label{moebius action}
Let $\h$ be a reproducing kernel Hilbert space consisting of analytic functions on the disk where the Moebius group $\m (\DD)$ acts by composition operators. Assume that $\h$ satisfies Hypothesis \ref{hipo}. Then the space $\grp$ is a closed submanifold of $\gr$, and a homogeneous space under the action of $\m(\DD)$.
\end{coro}
\begin{proof}
Since $\grp$ is finite dimensional, it suffices to show that for any fixed $a_0\in\DD$, the map
$$
\pi_{a_0}:\m(\DD)\to \grp\subset  \gr,
$$
has continuous local cross sections. Clearly, it has continuous sections when regarded as 
$$
\grp\sim\DD,
$$
 it is the quotient map $\m(\DD) / \mathbb{T}\sim\DD$.  By the above proposition, the map  $\pi_{a_0}$ has continuous local cross sections. It is clearly a smooth map, whose tangent map at the identity, regarded as a map from the tangent space of $\m(\DD)$ at the identity to the tangent space of $\gr$  at $\pi_{a_0}=P_{\z_{a_0}}$,  splits: it has finite dimensional nullspaces, and finite co-dimensional range. It follows from general facts from differential geometry, that $\pi_{a_0}$ is a submersion and that its image $\grp$ is a submanifold of $\gr$. 
\end{proof}

\begin{rem}
The metric $\Gamma$ defined on $X$  by means of the geodesic distance can be generalized. The geodesic distance can be used to define a pseudo-metric in the set of (finite) subsets of $X$ with the same cardinality. If $\aa=\{a_1,\dots,a_n\}$ and $\bb=\{b_1,\dots,b_n\}$, put
$$
\Gamma^n(\aa,\bb)=\Gamma^n_\h(\aa,\bb):=\|X_{\aa,\bb}\|=\arcsin\|P_{\z_\aa}-P_{\z_\bb}\|.
$$
This becomes a metric when any finite subset consisting of $n+1$ kernel functions is linearly independent. A sufficient condition is that the kernel function is strictly positive definite. 
\end{rem}

\section{Spectral and norm inequalities}\label{norm inequalities}

In this section we suppose that $\aa,\bb\subset X$ are disjoint with the same cardinality $n$ and satisfy the condition $\det \left( ( k_{b_i}(a_j))_{i,j=1}^n\right) \ne 0$. We also focus first on the generic part $\h_0$ of $\z_\aa$ and $\z_\bb$, which, as we saw in the preliminary section, is  a $2n$ dimensional space generated by $\{k_{a_i},k_{b_j}: 1\le i,j\le n\}$. Let us denote by $\z_\aa^0$ and $\z_\bb^0$ the intersections of $\z_\aa$ and $\z_\bb$ with $\h_0$. Then
$$
\z_\aa^0=\h_0\ominus \langle k_{a_i}: 1\le i\le n\rangle \ \hbox{ and } \  \z_\bb^0=\h_0\ominus \langle k_{b_j}: 1\le j\le n\rangle .
$$
Denote by $E_{\aa\bb}=P_{(\z_\aa^0)^\perp\parallel(\z_\bb^0)^\perp}$, i.e. the idempotent (non orthogonal projection) given by the direct sum decomposition
$$
\h_0 =\langle k_{a_i}: 1\le i\le n\rangle \dot{+} \langle k_{b_j}: 1\le j\le n\rangle, \ E_{\aa\bb}(f+g)=f.
$$
We recall in the following remark several facts of the theory of two subspaces:
\begin{rem}
Let $\s,\t$ be closed subspaces of a Hilbert space $\h$. Denote by $\alpha_0(\s,\t) \in [0,\pi /2]$ the Dixmier angle between $\s$ and $\t$, whose cosine is defined by
$$
C_0(\s,\t)=\sup\{|\langle f,g\rangle|: f\in\s, g\in\t, \|f\|=\|g\|=1\}.
$$
\begin{enumerate}
\item
This angle can be expressed in terms of orthogonal projections: $C_0(\s,\t)=\|P_\s P_\t\|$. Also it will  be useful for us to recall here that 
$C_0(\s , \t)<1$ if and only if $\s + \t$ is closed and $\s \cap \t = \{ 0 \}$. In this case $C_0(\s,\t)=C_0(\s^\perp,\t^\perp)$, and consequently, $\|P_\s P_\t\|=\|P_\s^\perp P_\t^\perp\|$ (see e.g. \cite{Deu}).

Furthermore, it is convenient to state now the following equivalences \cite{buckholtz}:
\begin{enumerate}
\item
$\s\dot{+}\t=\h$;
\item
$\s^\perp\dot{+}\t^\perp=\h$;
\item
$P_\s-P_\t$ is invertible.
\end{enumerate}
 In this case we have the following formula by Ando \cite{ando}: if $E=P_{\s\parallel\t}$ denotes the idempotent induced by the decomposition $\s\dot{+}\t=\h$, then
$$
E=P_\s(P_\s+P_\t)^{-1}.
$$
\item
If the subspaces $\s$ and $\t$ are in generic position, then there exists a unitary isomorphism between $\h$ and a product space $\ele\times\ele$ such that, with this isomorphism, the projections $P_\s$, $P_\t$ are unitarily equivalent to the projections
$$
\left(\begin{array}{cc} 1 & 0 \\ 0 & 0 \end{array} \right) \hbox{ and }  \left(\begin{array}{cc} C^2 & CS \\ CS & S^2 \end{array} \right),
$$
where $C=\cos(Z)$, $S=\sin(Z)$, for a positive operator $Z$ acting in $\ele$, with $\|Z\|\le \pi/2$ and trivial nullspace. These facts were proved by P. Halmos in \cite{halmos}.
\item
When $\s$ and $\t$ are in generic position, there exists a unique geodesic joining them.  It is of the form $\delta(t)=e^{itX}\s$, where $X=X_{\s,\t}$ is identified (by means of the above isomorphism) with the matrix
$$
X\simeq\left(\begin{array}{cc} 0 & iZ \\ -iZ & 0 \end{array} \right),
$$
where $Z$ is the positive operator given by Halmos' Theorem.
\end{enumerate}
\end{rem}
Note that  $\s=(\z_\aa^0)^\perp$ and $\t=(\z_\bb^0)^\perp$ are in generic position as subspaces of $\h_0$ and in direct sum, and therefore the facts above apply. For instance
$$
E_{\aa\bb}=P_{\z_\aa^0}^\perp(P_{\z_\aa^0}^\perp+P_{\z_\bb^0}^\perp)^{-1}.
$$
Denote by $X_{\aa\bb}$ the exponent of the unique geodesic joining $(\z_\aa^0)^\perp$ and $(\z_\bb^0)^\perp$. Then, it serves also as the exponent of the unique geodesic joining the orthogonal complements: $\z_\aa^0$ and $\z_\bb^0$.
If $A$ is a positive operator, we denote by  $\gamma_A$  the smallest spectral value of $A$ .
From the definition  of the Dixmier angle one obtains the following inequalities:
\begin{coro}\label{corolario42}
With the current notations,  for all $1\le i, j\le n$  we have
\begin{enumerate}
\item
$
\displaystyle{\frac{|\langle k_{a_i},k_{b_j}\rangle|}{\|k_{a_i}\|\|k_{b_j}\|}}\le \|P_{\z_\aa^0}P_{\z_\bb^0}\|=\|P_{\z_\aa^0}^\perp P_{\z_\bb^0}^\perp\|<1.
$
\item
$
\gamma_{X_{\aa\bb}}\le \arccos\left(\displaystyle{\frac{|k_{a_i}(b_j)|}{k_{a_i}(a_i)^{1/2}k_{b_j}(b_j)^{1/2}}}\right)=\arccos\left(\displaystyle{\frac{|\langle k_{a_i},k_{b_j}\rangle|}{\|k_{a_i}\|\|k_{b_j}\|}}\right).
$
\end{enumerate}
\end{coro}
\begin{proof}
The first assertion follows from the definition of the Dixmier angle, applied to the normalized elements $k_{a_i}$ and $k_{b_j}$. For the second assertion,  using Halmos representation for the pair $\s=(\z_\aa^0)^\perp$, $\t=(\z_\bb^0)^\perp$,  note that
$$
\|P_{\z_\aa^0}^\perp P_{\z_\bb^0}^\perp\|=\|P_{\z_\aa^0}^\perp P_{\z_\bb^0}^\perp P_{\z_\aa^0}^\perp\|^{1/2}=\|\cos^2(X_{\aa\bb})\|^{1/2}=\|\cos(X_{\aa\bb})\|,
$$
where the previous to last equality follows  from Halmos' representation:
$$
P_{\z_\aa^0}^\perp P_{\z_\bb^0}^\perp P_{\z_\aa^0}^\perp\simeq
\left(\begin{array}{cc} 1 & 0 \\ 0 & 0 \end{array} \right) \left(\begin{array}{cc} C^2 & CS \\ CS & S^2 \end{array} \right)\left(\begin{array}{cc} 1 & 0 \\ 0 & 0 \end{array} \right)=\left(\begin{array}{cc} C^2 & 0 \\ 0 & 0 \end{array} \right).
$$
Thus,
$$
\|P_{\z_\aa^0}^\perp P_{\z_\bb^0}^\perp P_{\z_\aa^0}^\perp\|=\|C^2\|=\|\cos(X_{\aa\bb})\|^2.
$$
Note that $X_{\aa\bb}$ is positive and invertible. The proof follows observing that since the cosine map is strictly decreasing, $\|\cos(X_{\aa\bb})\|$ is the cosine of the smallest eigenvalue of $X_{\aa\bb}$.
\end{proof}

With a similar argument, one obtains:
\begin{coro}\label{corolario43}
Let $f\in\z_\aa^0$ with $\|f\|=1$. Then for $1\le j\le n$,
$$
\arcsin\left(\frac{|f(b_j)|}{k_{b_j}(b_j)^{1/2}}\right)\le \|X_{\aa\bb}\|.
$$
\end{coro}
\begin{proof}
The subspaces $\z_\aa^0, (\z_\bb^0)^\perp$ of $\h_0$ are in direct sum. Indeed they have dimension $n$ and trivial intersection ($\z_\aa^0$ and $\z_\bb^0$ are in generic position). Therefore the facts in the previous remark hold for these subspaces. As in the previous result, consider  the pair $\s=\z_\aa^0$, $\t=(\z_\bb^0)^\perp$.
Then 
$$
C_0(\z_\aa^0,(\z_\bb^0)^\perp)=\|P_{\z_\aa^0}P_{\z_\bb^0}^\perp P_{\z_\aa^0}\|^{1/2}.
$$
In this case,
$$
P_{\z_\aa^0}P_{\z_\bb^0}^\perp P_{\z_\aa^0}\simeq\left(\begin{array}{cc} 0 & 0 \\ 0 & 1 \end{array}\right)\left(\begin{array}{cc} C^2 & CS \\ CS & S^2 \end{array}\right)\left(\begin{array}{cc} 0 & 0 \\ 0 & 1 \end{array}\right)=\left(\begin{array}{cc} 0 & 0 \\ 0 & S^2 \end{array}\right),
$$
and thus $C_0(\z_\aa^0,(\z_\bb^0)^\perp)=\|S^2\|^{1/2}=\sin(\|X_{\aa\bb}\|)$.
Now pick a function $f\in\z_\aa^0$ with $\|f\|=1$ and $\frac{1}{\|k_{b_j}\|}k_{b_j}\in (\z_\bb^0)^\perp$. Then
$$
\left|\left\langle f, \frac{k_{b_j}}{\|k_{b_j}\|}\right\rangle \right| \le C_0(\z_\aa^0,(\z_\bb^0)^\perp),
$$
and the proof follows.
\end{proof}
\begin{ejem}\label{44}
For instance, in the case $\h=H^2$, for any $1\le j_0\le n$,  one can take 
$$
f(z)=\displaystyle{(1-|b_{j_0}|^2)^{1/2}}k_{b_{j_0}}^H(z)B_{\aa}(z),
$$
where $B_{\aa}(z)$ is the Blaschke product with zeros $a_1,\dots,a_n$. That is, 
$
B_{\aa}(z)=\prod_{j=1}^n b_{a_j}(z),
$
where $b_0(z)=z$, and for $a_j\neq 0$, 
$$
b_{a_j}(z)= \frac{\bar{a}_j}{|a_j|} \frac{a_j-z}{1 - \bar{a}_j z}, \, \, \, z \in \DD.
$$
Indeed, clearly $\|f\|=1$ and $f(a_i)=0$ for $i=1,\dots,n$. In this case $\h_0=H^2\ominus B_{\aa\cup\bb} H^2$, which consists of rational functions of the  form
$$
\displaystyle{\frac{p(z)}{\prod_{i=1}^n (1-\bar{a}_iz) \prod_{j=1}^n (1-\bar{b}_jz)}},
$$
with $p(z)$ a polynomial of degree $\le 2n-1$. Clearly $f$ is of this form, taking 
$$p(z)=\displaystyle{{(1-|b_{j_0}|^2)^{1/2}}}\prod_{i=1}^n \frac{\bar{a}_i}{|a_i|} (a_i-z) \prod_{j\ne j_0}(1-\bar{b_j}z),$$
whenever $a_i \neq 0$ for all $i=1, \ldots, n$.

If we use the inequality of the above corollary for this function $f$ (for a given fixed $1\le j_0\le n$), we obtain
\begin{equation}\label{duplicado}
\arcsin(|B_\aa(b_{j_0})|)\le \|X_{\aa\bb}\|.
\end{equation}
\end{ejem}

\begin{rem}
The spectrum of the self-adjoint matrix $X_{\aa\bb}$ is related to the singular values of the idempotent matrix $E_{\aa\bb}$. For instance, since $P_{\z_\aa^0}^\perp$ and $E_{\aa\bb}$ project onto $(\z_\aa^0)^\perp$, one has that 
$$
P_{\z_\aa^0}^\perp=\left(\begin{array}{cc} 1 & 0 \\ 0 & 0 \end{array} \right) \ \hbox{ and }\ \  E_{\aa\bb}=\left(\begin{array}{cc} 1 & B \\ 0 & 0 \end{array} \right),
$$
where $B:\z_{\aa}^0\to (\z_{\aa}^0)^\perp$. Then, it is known that (see  \cite{ando}) 
$$
P_{\z_\aa^0}^\perp P_{\z_\bb^0}^\perp P_{\z_\aa^0}^\perp=\left(\begin{array}{cc} BB^*(1+BB^*)^{-1} & 0 \\ 0 & 0 \end{array} \right).
$$
Then $C^2=\cos^2(X_{\aa\bb})\simeq BB^*(1+BB^*)^{-1}$, i.e.,
$$
\sigma(X_{\aa\bb})=\left\{ \arccos\left( \frac{t}{\sqrt{1+t^2}}\right): t \hbox{ is a singular value of } E_{\aa\bb} \right\}. 
$$
Indeed,  $\sigma((BB^*)^{1/2})$ is the set of singular values of $E_{\aa\bb}$.
\end{rem}

\subsection{The case of the Hardy space}\label{estimates Hardy}
In the case $\h=H^2$, one can use the formulas by Adamyan, Arov and Krein \cite{aak} for the singular values of a Hankel operator to obtain expressions for the eigenvalues of the matrix $X_{\aa\bb}$. Let $P_+$ be the orthogonal projection of $L^2$ onto $H^2$, and $P_-$ be the orthogonal projection onto $H^2_-=L^2\ominus H^2$. Given $\varphi \in L^\infty=L^\infty(\mathbb{T})$, the Hankel operator with symbol $\varphi$ is defined by $H_\varphi : H^2 \to H^2_-$, $H_\varphi f=P_-(\varphi f)$, where $f \in H^2$.

If $\varphi\in L^\infty$, denote by $P_\varphi:L^2(\mathbb{T})\to \overline{\varphi H^2}$ the orthogonal projection. Recall that $\varphi H^2$ is closed if and only if $\varphi$ is an invertible function in $L^\infty$ (see e.g. \cite[Lemma 3.1]{acl18}). Let $\aa,\bb$ be finite disjoint sets of $\DD$. If we take $B_\aa$ and $B_\bb$  the finite Blaschke products with zeros in $\aa$ and $\bb$ respectively, then
$$
P_{B_\aa}=M_{B_\aa}P_+M_{\bar{B}_{\aa}} \ \hbox{ and } \ P_{B_\bb}=M_{B_\bb}P_+M_{\bar{B}_{\bb}}.
$$
 
\begin{prop}\label{proposicion46}
The eigenvalues $\lambda_0\ge \lambda_1\ge \ldots  \ge \lambda_{n-1}>0$ of $X_{\aa\bb}$ are given by
$$
\lambda_k=\arcsin(s_k(H_{B_{\bb}/B_{\aa}})),
$$
where $s_k(H_{B_{\bb}/B_{\aa}})$ denotes the $k$-th singular value (in decreasing order) of the Hankel operator $H_{B_{\bb}/B_{\aa}}$.
\end{prop}
\begin{proof}
For the Hardy space it holds that $\z_\aa=B_\aa H^2$ and $\z_\bb=B_\bb H^2$, then the generic part of the projections $P_{B_\aa}$ and $P_{B_\bb}$ is the $2n$ dimensional space $\h_0=\langle k_{a_i}, k_{b_j}: 1\le i,j\le n \rangle$. Note that on the space $\h_0$, the projections $P_{B_\aa}$ and $P_{B_\bb}$ reduce to $P_{\z_\aa^0}$ and $P_{\z_\bb^0}$. As before,
$$
P_{\z_\aa^0}^\perp P_{\z_\bb^0}P_{\z_\aa^0}^\perp\simeq \sin^2(X_{\aa\bb}).
$$
On the other hand, by the above remark the non trivial  eigenvalues  of $P_{B_\aa}^\perp P_{B_\bb} P_{B_\aa}^\perp$   coincide with those of 
$P_{\z_\aa^0}^\perp P_{\z_\bb^0}P_{\z_\aa^0}^\perp$. These, in turn, are the squares of the non trivial singular values of $P_{B_\aa}^\perp P_{B_\bb}$: note that 
$$
P_{B_\aa}^\perp P_{B_\bb}(P_{B_\aa}^\perp P_{B_\bb})^* =P_{B_\aa}^\perp P_{B_\bb} P_{B_\aa}^\perp.
$$
Also,
$$
P_{B_\aa}^\perp P_{B_\bb} =M_{B_\aa}P_+^\perp M_{\bar{B}_{\aa}} M_{B_\bb}P^+ M_{\bar{B}_\bb}=M_{B_\aa}P_+^\perp M_{B_\bb / B_\aa} P_+ M_{\bar{B}_\bb},
$$
which, since $M_{B_\aa}$ and $M_{B_\bb}$ are unitary operators, has the same singular values as
$$
P_+^\perp M_{B_\bb / B_\aa} P_+\simeq H_{B_\bb / B_\aa}.
$$
\end{proof}

The space of all bounded analytic functions on $\DD$ with the norm $\|f\|_\infty=\sup_{z \in \DD}|f(z)|$ is the Hardy space $H^\infty=H^\infty(\DD)$. Analogously as with  $H^2$, $H^\infty$ may be identified with the subspace of $L^\infty=L^\infty(\TT)$ given by $H^\infty=H^2 \cap L^\infty$. Furthermore, $H^\infty$ is a closed  subalgebra of $L^\infty$. 

Now we recall the following  theorem:
\begin{teo}{\rm (Adamjan, Arov, Krein  \cite{aak})}
Let $\varphi\in L^\infty$, and denote by $H_\varphi$ the Hankel operator with symbol $\varphi$. Let ${\cal R}_k$ denote the set of rational functions in $\mathbb{C}$ which tend to $0$ if $|z|\to \infty$, with poles in $\DD$ with total multiplicity $\le k$. Denote by $s_k(H_\varphi)$ the $k$-th singular value for $k\ge 0$, in non increasing order, repeated according multiplicity.
Then
$$
\begin{array}{ll} s_k(H_\varphi) & =\min
\{ \|H_\varphi-H_\psi\|: \mathrm{rank}(H_\psi)\le k\} \\  \ & 
=\mathrm{dist}(\varphi, {\cal R}_k+H^\infty) \\ \ & 
 =\min\{\|H_{B\varphi}\|: B \hbox{ is a Blaschke product of degree } \le k\}. \end{array}
$$
\end{teo}
\begin{ejem}
Using the last version of the above formula, one can obtain a lower estimate for the least eigenvalue $\gamma_{X_{\aa\bb}}$ of $X_{\aa\bb}$. Note that they are ordered, according to Proposition \ref{proposicion46}, 
$$
\|X_{\aa\bb}\|=\lambda_0\ge \lambda_1\ge \dots \lambda_{n-1}=\gamma_{X_{\aa\bb}}.
$$
We may suppose that $a_i \neq 0$, $i=1, \ldots ,n$. For a fixed $1\le j\le n$,   denote by $B_{\aa^j}=\prod_{i\ne j} \frac{\bar{a}_i}{|a_i|}\frac{z-a_i}{1-\bar{a}_i z}$, which is a product of degree $n-1$. Then by the above formula
$$
s_{n-1}(H_{B_\bb / B_\aa})\le \|H_{B_{\aa^j} B_\bb / B_\aa }\|=\|H_{B_\bb \frac{|a_j|}{\bar{a}_j}\frac{1-\bar{a}_jz}{z-a_j}}\|.
$$
This last norm equals the distance (in the $L^\infty(\mathbb{T })$-norm)  
$$
\mathrm{dist}\left(B_\bb \frac{|a_j|}{\bar{a}_j}\frac{1-\bar{a}_jz}{z-a_j},H^\infty\right)=\inf_{\varphi\in H^\infty} \left\|B_\bb \frac{|a_j|}{\bar{a}_j}\frac{1-\bar{a}_jz}{z-a_j}-\varphi \right\|.
$$
Denote $g(z)=B_\bb(1-\bar{a}_jz)$. Clearly, this infimum equals
$$
\inf_{\varphi\in H^\infty}\left\|\frac{g(z)}{z-a_j} -\varphi\right\|= \inf_{\varphi\in H^\infty}\left\|\frac{g(a_j)}{z-a_j}+\frac{g(z)-g(a_j)}{z-a_j}-\varphi\right\|=\inf_{\psi\in H^\infty}\left\|\frac{g(a_j)}{z-a_j}-\psi\right\|,
$$
because $\psi(z)=\frac{g(z)-g(a_j)}{z-a_j}-\varphi(z) \in H^\infty$. Note that since $\aa\cap\bb=\emptyset$, $g(a_j)\ne 0$. Thus 
$$
\inf_{\psi\in H^\infty}\left\|\frac{g(a_j)}{z-a_j}-\psi\right\|=|g(a_j)| \inf_{\psi\in H^\infty}\left\|\frac{1}{z-a_j}-\psi\right\|.
$$
\end{ejem}
For $a\in\DD$, denote by 
\begin{equation}\label{Nsuba}
N_a=\inf\left\{\left\|\frac{1}{z-a}-\varphi \right\| : \varphi\in H^\infty\right\}=\mathrm{dist}\left(\frac{1}{z-a}, H^\infty\right)=\|H_{\frac{1}{z-a}}\|
\end{equation}
Note that $N_a\le \frac{1}{1-|a|}$. However, this inequality may be strict (see the Lemma below).

In particular, in the above example, we have for any $j=1,\dots , n$,
$$
s_{n-1}(H_{B_\bb / B_\aa}) \le |B_\bb(a_j)|(1-|a_j|^2)N_{a_j} \le |B_\bb(a_j)|(1+|a_j|),  
$$
However, a better estimation can be obtained by the following lemma.

\begin{lem}\label{lemazo}
If $a\in\DD$, then $N_a=1$.
\end{lem}
\begin{proof}
Write $f_a(z)=\frac{1}{z-a}=\sum_{k\ge 0} a^kz^{-k-1}$. Note that $N_a$ is the norm of (the class of) $\frac{1}{z-a}$ in  $L^\infty / H^\infty\simeq (H^1_0)^*$, where $H^1_0$ denotes the subspace of functions in  $H^1$ with  mean value $0$ (in $\mathbb{T}$).
Denote by $m$ the normalized Lebesgue measure on $\mathbb{T}$. Then
$$
N_a=\sup\left\{ \, \left|\int_{\mathbb{T}} f_a g  dm\right| \, : \, g\in H^1_0, \, \|g\|_1=1 \, \right\}.
$$
Note that since $\int_{\mathbb{T}}g dm=0$, $g(z)=\sum_{k\ge 1} b_k z^k$, with $\sum_{k\ge 1}|b_k|=1$. 
Then
$$
\int_{\mathbb{T}}  f_a g   dm=\int_{\mathbb{T}}\left(\sum_{k\ge 0}a^k z^{-k-1}\right)\left( \sum_{j\ge 1}b_jz^j\right) dm =b_1+ab_2+a^2b_3+\dots  
$$
so that
$$
\left| \int_{\mathbb{T}}  f_a g  dm\right|\le \sum_{k\ge 1}|b_k||a|^{k-1}\leq 1.
$$
Taking functions of the form $g(z)=b_1 z$, $|b_1|=1$, we obtain that $N_a=1$.
\end{proof}

Combining these facts with Proposition \ref{proposicion46}, and Corollary \ref{corolario42}, one obtains:
\begin{coro}
With the current notations,
\begin{enumerate}
\item
$
\gamma_{X_{\aa\bb}}=\lambda_{n-1}\le  \min_{1\le j \le n} \arcsin |B_\bb(a_j)|(1-|a_j|^2).
$
\item
$
\|P_{\z_\aa}P_{\z_\bb}\|\ge\sqrt{1-|B_\bb(a_j)|^2(1-|a_j|^2)^2}, \, \, \, \, \, j=1, \ldots , n.
$
\end{enumerate}
\end{coro}
\begin{proof} 
Note that in this case $\|P_{\z_\aa}P_{\z_\bb}\|=\|P_{\z_\aa^0}P_{\z_\bb^0}\|$.
\end{proof}
Note that $|B_\bb(a_j)|=\prod_{k=1}^n \displaystyle{\frac{|b_k-a_j|}{|1-\bar{b_k}a_j|}}=\prod_{k=1}^n \rho(b_k,a_j)$.
Thus the above inequalities can be expressed in terms of the pseudodistance $\rho$. For instance,
$$ 
\lambda_{n-1}\le  \min_{1\le j \le n} \arcsin \prod_{k=1}^n \rho(b_k,a_j)(1-|a_j|^2).
$$

One can also obtain estimations for the eigenvalues of $X_{\aa\bb}$ by means of orthonormal bases of the model spaces. In the notation used in the theory of model spaces, if $u$ is an inner function, $\k_u=H^2\ominus u H^2$. Then, in our context
$$
\h_0=\k_{B_{\aa\cup\bb}} \ , \ \z_\aa^\perp=\k_{B_\aa} \ \hbox{ and } \ \z_\bb^\perp=\k_{B_\bb}.
$$
Recall the following result (\cite{tres}, \cite{sesentayuno}): if $\{u_k\}$ is a possibly finite sequence of inner functions such that $u=\prod_{k\ge 1}u_k$ exists, then
\begin{equation}\label{suma y producto}
\k_u=\k_{u_1}\oplus \bigoplus_{m\ge 2}\left( \prod_{k=1}^{m-1}u_k\right)  \k_{u_m}.
\end{equation}
\begin{rem}
We  use first this formula for the pair $u_1=B_\aa$, $u_2=B_\bb$
$$
\h_0=\k_\aa\oplus B_\aa \k_\bb.
$$
Note that in particular, this implies that the multiplication operator $M_{B_\aa}$ acting in $H^2$, which is an isometry, maps $\k_\bb$ onto $\h_0\ominus \k_\aa$. Then, we have that
\begin{equation}\label{mba}
M_{B_\aa} P_{\k_\bb}M_{B_\aa}^*=P_{\h_0}-P_{\k_\aa}.
\end{equation}
\end{rem}
Note that the singular values of  $P_{\k_\aa}P_{\k_\bb}$ are strictly between $0$ and $1$.
\begin{lem}\label{intercambio}
The singular values of $P_{\k_\aa}P_{\k_\bb}$,  are of the form $(1-s^2)^{1/2}$, where $s$ is a singular value of $P_{\k_\bb}M_{B_\aa}P_{\k_\bb}$ on the space $\k_\bb$.
\end{lem}
\begin{proof}
The singular values $s$ of $P_{\k_\aa}P_{\k_\bb}$ are the square roots of the eigenvalues of $P_{\k_\bb}P_{\k_\aa}P_{\k_\bb}$.
Using the formula (\ref{mba}), one has that 
$$
P_{\k_\bb}P_{\k_\aa}P_{\k_\bb}=P_{\k_\bb}-P_{\k_\bb}M_{B_\aa}P_{\k_\bb}M_{B_\aa}^*P_{\k_\bb},
$$
which can be regarded as an operator acting in $R(P_{\k_\bb})=\k_{\bb}$, and whose eigenvalues are of the form $1-\lambda$, where $\lambda$ is an eigenvalue of $P_{\k_\bb}M_{B_\aa}P_{\k_\bb}M_{B_\aa}^*P_{\k_\bb}$. These are, in turn, the squares of the singular values of $P_{\k_\bb}M_{B_\aa}P_{\k_\bb}$.
\end{proof}
One can also apply the formula (\ref{suma y producto}) to obtain an orthonormal basis for $\k_\bb=\k_{B_\bb}$, as is usual. Put
$$
\omega_1=\frac{k_{b_1}}{\|k_{b_1}\|} , \ \omega_2=B_{b_1}\frac{k_{b_2}}{\|k_{b_2}\|} \ , \ \omega_3=B_{b_2}B_{b_1}\frac{k_{b_3}}{\|k_{b_3}\|} , \dots
$$
This basis is called the Takenaka-Malmquist-Walsh basis in \cite{garcia}. 

\begin{lem}\label{triangular}
The operator $P_{\k_\bb}M_{B_\aa}P_{\k_\bb}$ acting in $\k_\bb$ is triangular in the basis $\{\omega_1. \dots, \omega_n\}$.
\end{lem}
\begin{proof}
Since multiplication by $B_{b_l}$ is isometric,
$$
\langle P_{\k_\bb}M_{B_\aa}P_{\k_\bb}\omega_{i+k}, \omega_i\rangle
=
\langle B_\aa \omega_{i+k},\omega_i\rangle
=
\langle B_\aa B_{b_1}\dots B_{b_{i+k-1}}B_\aa\frac{k_{b_{i+k}}}{\|k_{b_{i+k}}\|}, B_{b_1}\dots B_{b_{i-1}}\frac{k_{b_{i}}}{\|k_{b_{i}}\|}\rangle
$$
$$
=\langle B_\aa B_{b_{i}} \dots B_{b_{i+k-1}}B_\aa\frac{k_{b_{i+k}}}{\|k_{b_{i+k}}\|}, \frac{k_{b_i}}{\|k_{b_i}\|}\rangle=\frac{1}{\|k_{b_{i+k}}\|}\frac{1}{\|k_{b_i}\|}\langle B_\aa B_{b_{i}} \dots B_{b_{i+k-1}}B_\aa k_{b_{i+k}}, k_{b_i}\rangle=0,
$$
because $B_{b_i}(b_i)=0$.
\end{proof}
By a similar computation as above, the diagonal entries of $P_{\k_\bb}M_{B_\aa}P_{\k_\bb}$ in the basis $\{\omega_1, \dots, \omega_n\}$ are
\begin{equation}\label{entradas diagonales}
\lambda_j=\langle B_\aa \omega_j,\omega_j\rangle=\frac{1}{\|k_{b_j}\|^2}\langle B_\aa k_{b_j}, k_{b_j}\rangle= B_\aa(b_j)
\end{equation}

Using Weyl's inequalities for eigenvalues and singular values one obtains the following:
\begin{coro}
Suppose that the numbers $B_\aa(b_j)$, $1\le j \le n$  are arranged so that $|B_\aa(b_j)|$ are in decreasing order. Then, if $s_k(P_{\z_\aa^0}P_{\z_\bb^0})$, $0\le k \le n-1$ denote the singular values of $P_{\z_\aa^0}P_{\z_\bb^0}$ (arranged, as is usual, also in decreasing order), one has
\begin{enumerate}
\item
$$
\prod_{j=1}^m |B_\aa(b_j)|\le \prod_{j=1}^m \sqrt{1-s_{n-j}^2(P_{\z_\aa^0}P_{\z_\bb^0})}
$$
for $1\le m\le n$, with equality for $m=n$.
\item
$$
\sum_{j=1}^m |B_\aa(b_j)|\le \sum_{j=1}^m \sqrt{1-s^2_{n-j}(P_{\z_\aa^0}P_{\z_\bb^0})},
$$
for $1\le m\le n$.
\end{enumerate}
\end{coro}
\begin{proof}
First recall that the singular values of $P_{\k_\aa}P_{\k_\bb}=P_{\z_\aa^0}^\perp P_{\z_\bb^0}^\perp$ coincide with the singular values of $P_{\z_\aa^0}P_{\z_\bb^0}$. The statement follows by applying Weyl's inequalities \cite{horn} to the operator $P_{\z_\bb^0}M_{B_\aa} P_{\z_\bb^0}$. Note that the map $s\mapsto \sqrt{1-s^2}$ which (by Lemma \ref{intercambio}) is a bijection between the singular values of $P_{\z_\bb^0}M_{B_\aa}P_{\z_\bb^0}$ and the singular values of $P_{\z_\aa^0}P_{\z_\bb^0}$ reverses the order.
\end{proof}

In particular, for $m=1$:
\begin{coro}
With the current notations, i.e., $|B_\aa(b_1)|=\max_{1\le j\le n} |B_\aa(b_j)|$, one 
has
\begin{enumerate}
\item
$$
s_{n-1}(P_{\z_\aa^0}P_{\z_\bb^0})\le\sqrt{1-|B_\aa(b_1)|^2}.
$$
\item
$$
\|X_{\aa\bb}\|\ge \arcsin\left(\max_{1\le j,k\le n}\{|B_\aa(b_j)|,|B_\bb(a_k)|\}\right).
$$
\end{enumerate} 
\end{coro}
\begin{proof}
The first assertion is clear. For the second assertion note that $s_{n-1}^2(P_{\z_\aa^0}P_{\z_\bb^0})$ is the least eigenvalue of $P_{\z_\aa^0}P_{\z_\bb^0}P_{\z_\aa^0}=\cos^2(X_{\aa\bb})$, which, since $\cos$ is decreasing, equals $\cos^2(\|X_{\aa\bb}\|)$. Then
$$
\cos^2(\|X_{\aa\bb}\|)=s_{n-1}^2(P_{\z_\aa^0}P_{\z_\bb^0})\le 1-|B_\aa(b_1)|^2.
$$
Thus, 
$$
\|X_{\aa\bb}\|\ge \arccos(\sqrt{1-|B_\aa(b_1)|^2})=\arcsin(|B_\aa(b_1)|).
$$
The numbers $B_\aa(b_j)$ were arranged so that $|B_\aa(b_1)|$ has maximum modulus. Clearly, the roles of $\aa$ and $\bb$ are symmetric, thus the inequality follows.
\end{proof}
\begin{rem}
The inequality 2.  of the above corollary, was obtained in (\ref{duplicado}), in Example \ref{44}, by other means.
\end{rem}

\section{Infinite zero sets in the Hardy space}\label{infinite sets}

In this section we only consider the case $\h=H^2$ and  subsets $\aa=\{a_k: k\ge 1\}$, $\bb=\{b_j: j\ge 1\}$ of the unit disk which are infinite. Recall that the zeros $\{ a_k \}$ of a function $f \in H^2$ must satisfy Blaschke condition $\sum_{k\ge 1} (1-|a_k|)<\infty$. This condition guarantees the convergence on compact subsets  of the infinite Blaschke product  $B(z)$ with zeros $\{ a_k \}$ given by 
$
B(z)=\prod_{k=1}^\infty b_{a_k}(z),
$
where $b_0(z)=z$, and for $a_k\neq 0$, $
b_{a_k}(z)= \frac{\bar{a}_k}{|a_k|} \frac{a_k-z}{1 - \bar{a}_k z}, \, \, \, z \in \DD.
$
We still denote by $B_\aa$ and $B_\bb$ the (now infinite) Blaschke products with zeros $\aa$, $\bb$, respectively (Note: $\aa$, $\bb$ are regarded as {\it sets} and not sequences, so all zeros are simple zeros). 
\begin{rem}
If we denote $\z_\aa=\{f\in H^2: f(a_k)=0, k\ge 1\}$, $\z_\bb=\{f\in H^2: f(b_j)=0, j\ge 1\}$, and note that the Blaschke products $B_\aa$ and $B_\bb$ have simple zeros, then it clearly follows that $\z_\aa=B_\aa H^2$ and $\z_\bb=B_\bb H^2$.  Suppose additionally that $\aa\cap\bb=\emptyset$. Then, we have
$$
\z_\aa\cap \z_\bb=B_{\aa\cup\bb} H^2\ \  \hbox{ and  }\ \ \z_\aa^\perp\cap\z_\bb^\perp=\{0\}.
$$
The first fact is clear. For the second, note that $(\z_\aa^\perp\cap\z_\bb^\perp)^\perp= \langle \z_\aa \vee \z_\bb\rangle$. Since $\aa\cap\bb=\emptyset$, $B_\aa$ and $B_\bb$ are co-prime inner functions, thus $\z_\aa\vee\z_\bb$ generates $\h$.
\end{rem}
So it  remains to examine $\z_\aa\cap\z_\bb^\perp$ and $\z_\aa^\perp\cap\z_\bb$. Denote by $\lim \aa$ the limit  set of $\aa$, i.e. $\lim \aa =\overline{\aa}\setminus\aa$. Also $\lim \aa$ is usually known as the support of the Blaschke product $B_\aa$. Note that since $\aa, \bb$ accumulate only at $\mathbb{T}$: $\lim\aa,\lim\bb\subset \mathbb{T}$.

\begin{prop}\label{different supports}
Suppose that $\aa,\bb$ are disjoint infinite subsets of $\DD$, satisfying Blaschke's condition, and such that 
$$
\lim\aa\not\subset \lim\bb \ \ \hbox{ and } \ \ \lim\bb\not\subset\lim\aa.
$$
Then $\z_\aa\cap\z_\bb^\perp=\z_\aa^\perp\cap\z_\bb=\{ 0 \}$.
\end{prop}
\begin{proof}
Note that since $B_\bb, B_\aa \in H^\infty$, 
$$
T_{B_\bb / B_\aa}=T_{\bar{B}_\aa B_\bb}=T_{\bar{B}_\aa}T_{B_\bb}=T_{\bar{B}_\aa}M_{B_\bb},
$$
where $M_{B_\aa}, M_{B_\bb}$ denote the multiplication operators (by $B_\aa$ and $B_\bb$), which are isometries of $H^2$ onto $B_\aa H^2=\z_\aa$ and $B_\bb H^2=\z_\bb$,  respectively. As is usual notation, $T_\varphi=P_{H^2} M_\varphi|_{H^2}$ denotes the Toeplitz operator with symbol $\varphi\in L^\infty$.
Thus, 
$$
N(T_{B_\bb / B_\aa})=\{f\in\z_\bb: T_{\bar{B}_\aa} f=0\}=\z_\bb\cap\z_\aa^\perp,
$$
because $T_{\bar{B}_\aa}f=0$  if and only if $0=\langle\bar{B}_\aa f,g\rangle=\langle f, B_\aa g\rangle$ for all $g\in H^2$, i.e. $f\in (B_\aa H^2)^\perp=\z_\aa^\perp$.

On the other hand, the fact that there exists $z_0\in\lim\bb\setminus\lim\aa$, implies, by a result by Lee and Sarason (Theorem 2 in \cite{leesarason}), that $T_{\bar{B}_\bb B_\aa}$ has dense range. Or equivalently, that $T_{\bar{B}_\bb B_\aa}^*=T_{\bar{B}_\aa B_\bb}$ has trivial nullspace. Thus $\z_\bb\cap\z_\aa^\perp=\{0\}$. By a similar argument, using that $\lim\aa\setminus\lim\bb\ne\emptyset$, one obtains that $\z_\aa\cap\z_\bb^\perp$ is trivial.
\end{proof}

\begin{coro}
Let $\aa$ and $\bb$ be infinite disjoint sets of $\DD$ which satisfy Blaschke's condition and such that $\lim\aa\not\subset \lim\bb$  and $\lim\bb\not\subset\lim\aa$. Then there exists a unique minimal geodesic of the Grassmann manifold which joins $\z_\aa$ and $\z_\bb$. That is, there exists a unique self-adjoint operator $X_{\aa\bb}$ satisfying that $X_{\aa\bb}\z_\aa\subset \z_\aa^\perp$ and $X_{\aa\bb}\z_\bb\subset \z_\bb^\perp$ with $\|X_{\aa\bb}\|\le \pi/2$, and 
$$
e^{iX_{\aa\bb}}\z_\aa=\z_\bb.
$$
\end{coro}

%\subsection{The same one point support}

We now examine the case of two infinite sets $\aa$, $\bb$ having the same accumulation points.  
We  suppose that both sets have only one accumulation point. For instance, let us assume that $\lim \aa =\lim \bb=1$.
Before we give our main result, we need to recall the following facts (see, for instance, \cite{douglas}).

\begin{rem}\label{sarason alg}
Let $C$ denote the algebra of continuous functions on $\TT$. The \textit{Sarason algebra} is the following algebraic sum
$$
H^\infty + C =\{ \, f+g  \, : \, f \in H^\infty, \, g \in C  \,   \}.
$$
It is known that $H^\infty+C$ is  a closed subalgebra of $L^\infty$. The harmonic extension $\hat{h}$ to   $\DD$ of a function $h \in \hc$ is well-defined, and  plays a fundamental role in the characterization of invertible functions in this algebra. For $h \in \hc$ and $0<r<1$,  set $h_r(e^{it})=\hat{h}(re^{it})$. Then $h$ is invertible in $\hc$ if and only if there exist $\delta, \epsilon >0$ such that $|h_r(e^{it})|\geq \epsilon$ for $1-\delta<r<1$ and $e^{it} \in \TT$. For instance, a Blaschke product is invertible in $\hc$ if and only if it is finite. 

This criterion of invertibility allows one to define the index of an invertible function in $\hc$.
For a  non-vanishing function $h \in C$, let $\text{ind}(h)\in \ZZ$ be the index (or winding number) of $h$ around $z=0$. 
For $h$  invertible in $\hc$, set $\text{ind}(h)=\lim_{r \to 1^-} \text{ind} (h_r)$. This index is stable by small perturbations and it is an homomorphism of the invertible functions in $\hc$ onto the group of integers.
\end{rem}

In their study of division in the algebra $H^\infty + C$, Guillory and Sarason \cite{guillorysarason} stated without proof that there exist two Blaschke products which are co-divisible in $H^\infty + C$. Below we give a proof of this fact, which combined with  well-known results on  Toeplitz operators allows us to construct examples of existence and non existence of geodesics between $\z_\aa$ and $\z_\bb$ in the case where $\lim \aa=\lim \bb=\{ 1\}$.

\begin{teo}\label{equal supports1}
Given an integer $m\geq 0$, there are two disjoint infinite sets $\aa, \bb \subseteq \DD$ such that 
\begin{enumerate}
\item[i)] $\aa, \bb$ satisfy Blaschke condition;
\item[ii)] $\lim \aa = \lim \bb=\{1\}$;
\item[iii)]   $\dim \z_\aa ^\perp \cap \z_\bb=0$ and $\dim \z_\aa \cap \z_\bb^\perp=m$.
\end{enumerate}
In this case, there exists a geodesic in $\mathrm{Gr}(H^2)$ joining $\z_\aa$ and $\z_\bb$ if and only if $m=0$. 
\end{teo}
\begin{proof}
Take a set $\aa=\{ \, a_k \, : \, k \geq 1\, \}\subseteq \DD\setminus \{0\}$ satisfying Blaschke condition and $\lim \aa=1$. Consider the Blaschke product 
$B_\aa$ with simple zeros given by the set $\aa$.  Next take the set $\bb=\{ \, b_k \, : \, k \geq 1 \, \}$ satisfying the following conditions:
\begin{itemize}
\item[i)] $b_k=a_k + \epsilon_k \in \DD\setminus \{ 0\}$, $\epsilon_k >0$, $\forall k \geq 1$;
\item[ii)] $\epsilon_k < \min \{ \,   \delta(a_k, \aa\setminus \{ a_k\}) , \, \delta^2(a_k,\TT) \, , \epsilon_{k-1} \}$,
\end{itemize}
where $\delta$ is the usual distance to a given set on the complex plane. Notice that the second condition above guarantees that $\aa \cap \bb=\emptyset$, $\lim \bb=\{ 1\}$ and $b_k \neq b_j$ if $k\neq j$. Denote by $B_\bb$ the Blaschke product 
with simple zeros given by the set $\bb$. We need to consider the following Blaschke factors:
$$
a_k(z)= \frac{\bar{a}_k}{|a_k|} \frac{a_k-z}{1 - \bar{a}_k z}; \, \, \, \, \, \, \, b_k(z)= \frac{\bar{b}_k}{|b_k|} \frac{b_k-z}{1 - \bar{b}_k z}.
$$
Claim:
$$\sup_{\theta \in [0,2\pi]}\left| \frac{b_k}{a_k}(e^{i\theta}) - 1 \right| \to 0.$$
To prove it, we split the function $b_k/a_k$ into three factors:
$$
 \frac{b_k}{a_k}(e^{i\theta})= \frac{\bar{b}_k}{|b_k|}\frac{|a_k|}{\bar{a}_k} \times \frac{b_k - e^{i\theta}}{a_k - e^{i\theta}} \times \frac{1-\bar{a}_k e^{i\theta}}{1-\bar{b}_ke^{i\theta}}:= F_1 \times F_2 \times F_3
$$
Using that $\epsilon_k \to 0$, we have that $F_1 \to 1$. The next factor can be estimated as follows:
\begin{align*}
| F_2 - 1| & = \left|  \frac{b_k - e^{i\theta}}{a_k - e^{i\theta}}  -1   \right| = \left| \frac{\epsilon_k}{a_k - e^{i\theta}}\right| 
 \leq \frac{\delta^2(a_k,\TT)}{\delta(a_k,\TT)}=\delta(a_k,\TT)\to 0.
\end{align*}
Similarly, we proceed with the third factor
\begin{align*}
|F_3 - 1 | & = \left|\frac{1-\bar{a}_k e^{i\theta}}{1-\bar{b}_ke^{i\theta}}\right|=   \frac{\epsilon_k}{|1- \bar{b}_k e^{i\theta}|} \\
& =   \frac{\epsilon_k}{| b_k - e^{i\theta}|} \leq \frac{\epsilon_k}{|e^{i\theta}-a_k|-\epsilon_k} \leq \frac{\delta(a_k,\TT)}{1-\delta(a_k,\TT)}\to 0.
\end{align*}
This proves our claim. 
Denote by $\| \, \, \|_\infty$ the uniform norm on $\TT$. Using that the functions $a_k$ are unimodular, and passing to adequate subsequences, we can obtain that 
\begin{equation}\label{est}
\| b_k  - a_k \|_\infty \leq \frac{1}{2^k}, \, \, \, \, \forall k \geq 1.
\end{equation}
From now on, $B_\aa$ and $B_\bb$ are the Blaschke products corresponding to the chosen subsequences. 

Next we need to recall how to estimate products in terms of sums (see e.g. \cite[Lemma 2.1]{gorkinmortini}): let $n \in \mathbb{N}$, and let $x_j$ and $y_j$ be complex numbers  with $|x_j|\leq 1$, $|y_j|\leq 1$. Then
\begin{equation}\label{est1}
\left| \prod_{j=1}^n x_j  -  \prod_{j=1}^n y_j  \right| \leq \sum_{j=1}^n |x_j - y_j |.
\end{equation}
Take the finite Blaschke products
$
B_\aa^{(n)}=\prod_{k=1}^n a_k , \, \, \, \, \, \,  B_\bb^{(n)}=\prod_{k=1}^n b_k.
$
It is well-known that these finite Blaschke products  converge in $H^2$ to the corresponding infinite product (see \cite{hoffman}). Since $B_\aa^{(n)}$, $B_\bb^{(n)}$, $B_\aa$, $B_\bb$ are unimodular functions, it follows that $B_\bb^{(n)}/B_\aa^{(n)}$ also converges to $B_\bb/B_\aa$ in $H^2$. Therefore there is a subsequence $\{ n_l \}$ such that $B_\bb^{(n_l)}/B_\aa^{(n_l)}$ converges pointwise to $B_\bb/B_\aa$ almost everywhere on $\TT$.  Using the estimates (\ref{est}) and (\ref{est1}), we get that almost everywhere on $\TT$, the following holds:
\begin{align*}
\left| \frac{B_\bb}{B_\aa}(e^{i\theta})  -  \frac{B_\bb^{(n_l)}}{B_\aa^{(n_l)}}(e^{i\theta})\right| & 
= \left| \prod_{k=n_l+1}^\infty \frac{b_k}{a_k}(e^{i\theta}) - 1   \right| 
 = \left|  \prod_{k=n_l+1}^\infty  b_k(e^{i\theta})  -   \prod_{k=n_l+1}^\infty a_k(e^{i\theta})  \right| \\
& \leq \sum_{k=n_l+1}^\infty |  b_k(e^{i\theta}) - a_k(e^{i\theta})   | \leq \sum_{k=n_l+1}^\infty \frac{1}{2^k}\xrightarrow[l \to \infty]{}0 
\end{align*}  
Thus, we have proved that for all $\epsilon>0$, there is $N\geq 1$ such that 
\begin{equation}\label{est2}
\| B_\bb/B_\aa  -  B_\bb^{(N)}/B_\aa^{(N)} \|_\infty < \epsilon.
\end{equation}
Notice that $B_\bb^{(N)}/B_\aa^{(N)}\in H^\infty + C$. This follows by recalling that finite Blaschke products are invertible in $H^\infty + C$, and $B_\bb^{(N)} \in H^\infty$. Therefore we have   $\mathrm{dist}(B_\bb/B_\aa,H^\infty + C)=0$, and consequently,  $B_\bb/B_\aa \in H^\infty + C$.  
One can see that the same estimates prove that $B_\aa/B_\bb \in H^\infty + C$. Indeed, note that the estimates depend on the difference of the functions
$a_k(e^{i\theta})$ and $b_k(e^{i\theta})$. Thus, $B_\bb/B_\aa$ is an invertible function in $H^\infty + C$.

Now we recall some characterizations of invertible Toeplitz operators (see e.g. \cite{douglas}). A result by Douglas states that given a function $f \in H^\infty + C$, then the Toeplitz operator $T_f$ is Fredholm if  and only if the function $f$ is invertible in $H^\infty + C$. Furthermore, $\text{ind}(T_f)=-\text{ind}(f)$, where the last index is that of invertible functions in $H^\infty +C$ (Remark \ref{sarason alg}). In addition, it is a well-known fact that 
for a function $f \in L^\infty$ such that $T_f$ is Fredholm, then $T_f$ is invertible if and only $\text{ind}(T_f)=0$. 
Returning to our example, we have that $T_{B_\bb/B_\aa}$ is a Fredholm operator satisfying
\begin{align*}
\text{ind}(B_\bb/B_\aa)=-\text{ind}(T_{B_\bb/B_\aa}) & =  \dim N ( T_{B_\aa/B_\bb}) - \dim N(T_{B_\bb/B_\aa}) \\
& =   \dim \z_\aa \cap \z_\bb^\perp - \dim \z_\aa^\perp \cap \z_\bb :=r
\end{align*}
 If necessary, that is, when $r \neq m$, we may modify $\aa$ or $\bb$ to get that $\text{ind}(B_\bb/A_\aa)=m$.
To this end note that Coburn's lemma gives that either $N ( T_{B_\aa/B_\bb})\simeq \z_\aa \cap \z_\bb^\perp =\{ 0\}$ or  $N(T_{B_\bb/B_\aa}) \simeq \z_\aa^\perp \cap \z_\bb=\{ 0\}$. For instance, let us suppose that  $\z_\aa^\perp \cap \z_\bb=\{ 0\}$; the other case can be treated similarly. Then consider a finite set $\textbf{c}=\{ \,  c_k \, : \, k=1, \ldots |m-r| \, \}$ such that $\textbf{c}\cap \aa \cap \bb=\emptyset$. If $r>m$, take $\aa_1=\aa \cup \textbf{c}$ and $\bb_1=\bb$, and if $r<m$, take $\aa=\aa_1$ and $\bb_1=\bb \cup \textbf{c}$. Therefore we have
$\text{ind}(B_{\bb_1}/B_{\aa_1})=m \geq 0$, and thus again by Coburn's lemma, it must be $\dim \z_{\aa_1} \cap \z_{\bb_1}^\perp \simeq \dim N(T_{B_{\aa_1}/B_{\bb_1}})=m$ and $\dim \z_{\aa_1}^\perp \cap \z_{\bb_1}\simeq \dim N(T_{B_{\bb_1}/B_{\aa_1}})=0$.
\end{proof}

Moreover, one can construct examples in which $\dim \z_\aa \cap \z_\bb^\perp=\infty$ and $\dim \z_\aa^\perp \cap \z_\bb=0$. To do this, it will be convenient  to recall the following facts.

%Now we examine the case of subspaces of the Hardy space $H^2$ defined  by sequences with infinitely many points. If no restriction is imposed on  $\aa \cap \bb$, one can easily construct examples such that $\z_\aa$ and $\z_\bb$ cannot be joined by a geodesic in $Gr(H^2)$. For instance, take $\aa$ a subsequence of $\bb$. Thus, the main case studied in this section is that of disjoint sequences $\aa$ and $\bb$ consisting in infinitely many points  accumulating to one point. We may suppose that the accumulation point is $z=1$. 

\begin{rem} 

\noindent

\begin{enumerate}
\item
 For $a>0$,  consider the singular inner function 
$$
\psi_a(z)=\exp(a (z+1)/(z-1)).
$$
A Blaschke product $B$ satisfying the condition $\psi_a \bar{B} \in H^\infty + C$ for all $a>0$, is called a \textit{Koosis function}.  
Such kind of functions exist in abundance. For example, if  $B$  is a Blaschke product with real simple zeros accumulating at $1$, then it is a Koosis function. We refer to \cite{leesarason} and the references therein.

\item
Let $\psi$ be an inner function. Frostman's theorem states that for all $|\gamma|<1$, except possibly for a set of capacity zero, the function 
$$
A_{\psi, \, \gamma}(z)=\frac{\psi(z)- \gamma }{1-\bar{\gamma}\psi(z)}
$$
is a Blaschke product. For instance, sets with zero capacity must have zero planar Lebesgue measure (see e.g. \cite{garnet}). 

\end{enumerate}
\end{rem}

\begin{teo}\label{equal supports2}
Let $B$ be a Koosis function with real simple zeros $\bb=\{ \, b_k \, : \, k \geq 1 \,\}$ such that $\lim \bb=\{1\}$. For $a>0$, consider the set
$$
\aa=\left\{ \, \frac{\alpha + 2k\pi - i(a+ \ln(|\gamma|))}{\alpha + 2k\pi + i(a- \ln(|\gamma|))} \, : \, k \in \ZZ \, \right\}, \, \, \, \, \alpha=\arg\left(\frac{\gamma}{|\gamma|}\right),
$$
for every $|\gamma | <1$ such that $A_{\psi_a, \, \gamma}$ is a Blaschke product and $\alpha \neq k \pi$, $k \in \ZZ$. Then  $\aa, \bb$ are  disjoint infinite subsets of $\DD$ such that 
\begin{enumerate}
\item[i)] $\aa, \bb$ satisfy Blaschke condition;
\item[ii)] $\lim \aa = \lim \bb=\{ 1 \}$;
\item[iii)]   $\dim \z_\aa ^\perp \cap \z_\bb=\infty$ and $\dim \z_\aa \cap \z_\bb^\perp=0$.
\end{enumerate}
In particular, there is no geodesic in $\mathrm{Gr}(H^2)$ joining the subspaces $\z_\aa$ and $\z_\bb$.
\end{teo}
\begin{proof}
First note that the condition $\alpha \neq k \pi$, $k \in \ZZ$ implies that $\Im \left(\frac{\alpha + 2k\pi - i(a+ \ln(|\gamma|))}{\alpha + 2k\pi + i(a- \ln(|\gamma|))}\right)\neq 0$, and thus, we get $\aa \cap \bb =\emptyset$. A direct computation shows that the zeros of $A:=A_{\psi_a , \, \gamma}$ are given by the sequence $\aa$.  Indeed, we can find them by noting that $\psi_a(z)=\gamma$, $z \in \DD$, if and only if $e^{iaw}=\gamma$, where 
$w=i(1+z)/(1-z)$ maps $\DD$ onto the upper half-plane $\CC_+$.  Now  a simple computation gives the solutions $w_k=u_k+iv$, where $\alpha=\arg(e^{av}\gamma)$,
$u_k=(\alpha + 2k\pi)/a$, $k\in \ZZ$ and $v>0$ is uniquely determined by $e^{av}=|\gamma|^{-1}$.  Thus, the zeros of $A$  are given by
$a_k=(w_k-i)/(w_k+i)$, $k \in \ZZ$. Note also that the zeros of $A$ are simple and $a_k \to 1$.

%Recalling that $\z_\aa=AH^2$ and $\z_\bb=BH^2$, then $\z_\aa ^\perp \cap \z_\bb = \{ 0\}$ if and only if for any function 
%$f \in \z_\aa^\perp=:K_A$ such that $f(b_k)=0$ for all $k$, we must have $f\equiv 0$. This is usually referred to as saying that  \textit{$\bb$ is a set of uniqueness for $K_A$}.  

As before, denote by $\mathcal{K}_{\psi_a}:=(\psi_aH^2)^\perp$, $\mathcal{K}_A:=(AH^2)^\perp=\z_\aa^\perp$, and recall that there is a unitary map defined by
$$
F:\mathcal{K}_{\psi_a} \to \mathcal{K}_A, \, \, \, \, Fh= (1-|\gamma|^2)^{1/2}\frac{h}{1-\bar{\gamma}\psi_a}.
$$
Further, note that $h(b_k)=0$, for all $k\geq 1$ if and only if $(Fh)(b_k)=0$ for all $k \geq 1$.
Therefore $K_{\psi_a} \cap \z_\bb \simeq \z_\aa^\perp \cap \z_\bb$. 
But if $B$ is a Koosis function, then $\dim \z_\aa ^\perp \cap \z_\bb =\dim N(T_{B/\psi_a })=\infty$ (see \cite[Thm. 4]{leesarason}).
Then by Coburn's Lemma, we have  $\z_\aa \cap \z_\bb^\perp \simeq N(T_{A/B})=0$.
\end{proof}

\begin{rem}

\noindent
\begin{enumerate}
\item
The kernel of the operator $T_{B/\psi_a}$ is related with a classical problem of completeness of exponentials as follows: set 
$$
\lambda_k=i \frac{1+b_k}{1-b_k},
$$
then by \cite[Lemma 7]{leesarason}, we have $N(T_{B/\psi_a})=\langle e^{\lambda_k x} \, : \, k \geq 1 \rangle ^\perp$, where the the orthogonal is considered as a subset of $L^2(0,a)$. 
\item
As a more concrete example, we observe that in \cite[p. 539]{leesarason} the following result by L. Schwartz is considered as a particular case of the preceding item: if $\{ \lambda_k\}$ is a sequence of pure imaginary numbers satisfying $\sum 1/|\lambda_k|< \infty$, then $\langle e^{\lambda_k x} \, : \, k \geq 1 \rangle \neq L^2(0,a)$ for all $a>0$.
\end{enumerate}
\end{rem}

\subsection{Compactness conditions} 

Under the same assumptions and notations ($\h=H^2$, $\aa$, $\bb$ infinite disjoints subsets of $\DD$ satisfying Blaschke condition), let us examine the condition that $P_{\k_\bb}P_{\z_\aa^0}=P_{\k_\bb}P_{\h_0\ominus\k_\aa}$ is compact. 
First note that 
$$
P_{\k_\bb}P_{\z_\aa^0}=P_{\k_\bb}M_{B_\aa}P_{\k_\bb}M_{B_\aa}^* 
$$
is compact in $\h_0=\k_{\aa\cup\bb}$ if and only if $P_{\k_\bb}M_{B_\aa}P_{\k_\bb}$ is compact in $\k_\bb$. Indeed
$M_{B_\aa}$ is an isometry in $\h_0$.
\begin{rem}\label{takenaka}
Note also that Lemma  \ref{triangular} holds in this (infinite) context: $P_{\k_\bb}M_{B\aa}P_{\k_\bb}$ is triangular in the Takenaka-Malmquist-Walsh basis
$$
\omega_1=\frac{k_{b_1}}{\|k_{b_1}\|} , \ \omega_2=B_{b_1}\frac{k_{b_2}}{\|k_{b_2}\|} \ , \ \omega_3=B_{b_2}B_{b_1}\frac{k_{b_3}}{\|k_{b_3}\|} , \dots
$$
In  particular, the eigenvalues of $P_{\k_\bb}M_{B_\aa}P_{\k_\bb}$ are $\{B_\aa(b_j): j\ge 1\}$.
\end{rem}

Then,  we have the following elementary necessary conditions:
\begin{prop}\label{prop56}
Suppose that  $P_{\k_\bb}P_{\z_\aa^0}$ is compact, then
\begin{enumerate}
\item
 $B_\aa(b_j)\to 0$;
\item
$\lim \bb\subset\lim\aa$.
\end{enumerate}
\end{prop}
\begin{proof}
Due to the above considerations, only the second assertion needs a proof.  Let $z_0\in\mathbb{T}$ be a limit point of $\bb$. Then, there exists a subsequence $\bb'=\{b_{j_1},b_{j_2},\dots\}$ such that $b_{j_k}\to z_0$. Since 
$\k_{\bb'}\subset \k_{\bb}$, we have that $P_{\k_{\bb'}}\le P_{\k_\bb}$. 
Note that $P_{\k_\bb}P_{\z_\aa^0}$ is compact if and only if  
$$
(P_{\k_\bb}P_{\z_\aa^0})^*P_{\k_\bb}P_{\z_\aa^0}=P_{\z_\aa^0}P_{\k_\bb}P_{\z_\aa^0}
$$
is compact. Then the fact that
$$
0\le P_{\z_\aa^0}P_{\k_{\bb'}}P_{\z_\aa^0}\le P_{\z_\aa^0}P_{\k_\bb}P_{\z_\aa^0}
$$
implies that $P_{\z_\aa^0}P_{\k_{\bb'}}P_{\z_\aa^0}$ is compact, and thus $P_{\k_{\bb'}}P_{\z_\aa^0}$ is compact.
Here we have used the following elementary fact: if $0\le A\le B$ in $\b(\h)$ and $B$ is compact, then $A$ is compact. Indeed, consider the projection $\pi:\b(\h)\to \b(\h)/K(\h)$ onto the Calkin algebra. Since it is a $*$-homomorphism, it preserves the order of operators, then $0\le\pi(A)\le\pi(B)=0$, and then $\pi(A)=0$.

Thus we may apply the first assertion of this Proposition to $\aa$ and $\bb'$: $B_\aa(b_{j_k})\to 0$. By Lemma 2 in \cite{leesarason}, we have that $z_0\in\lim\aa$.
\end{proof}
\begin{coro}
If both $P_{\k_\bb}P_{\z_\aa^0}$ and $P_{\k_\aa}P_{\z_\bb^0}$ are compact, then $\lim\aa=\lim\bb$.
\end{coro}

\begin{rem}
Note that $P_{\k_\bb}P_{\z_\aa^0}=P_{\k_\bb}M_{B_\aa}P_{\k_\bb}M_{B_\aa}^*$ is compact (or more generally, belongs to a $p$-Schatten class) if and only if $P_{\k_\bb}M_{B_\aa}P_{\k_\bb}$ is compact (resp. belongs to a $p$-Schatten class). Also it is clear that these assertions hold if and only inf they hold for
$$
P_{\k_\bb}M_{B_\aa}|_{\k_\bb},
$$
which is what in the literature is known as a {\it truncated Toeplitz operator} (with symbol $B_\aa$) \cite{tto}. In Theorem 3 of \cite{lopatto} it was shown that if $\theta$ is analytic and the set $\bb$ is an interpolating sequence, meaning that they satisfy the Carleson condition
$$
\inf_{j\ge 1} \prod_{k\ne j} \left|\frac{b_j-b_k}{1-\bar{b}_jb_k}\right| >0,
$$
then
$$
P_{\k_\bb} M_\theta |_{\k_\bb}
$$
\begin{enumerate}
\item
is compact if and only if $\theta(b_j)\to 0$;
\item
belongs to the $p$-Schatten class ($1\le p <\infty$) if and only if $\{\theta(b_j)\}\in\ell^p$.
\end{enumerate}

Proposition \ref{prop56} shows that if $\theta=B_\aa$ is a Blaschke product, condition 1. is necessary for compactness, for arbitrary Blaschke products $B_\bb$ (satisfying Blaschke condition).  Note that, again in this context,   condition 2. is also necessary. Indeed, it is well known that the pinching map of a given orthonormal basis, which sends an operator to the diagonal operator whose entries are the diagonal entries of the original operator, preserves the $p$-Schatten classes \cite{horn}. Thus, if $P_{\k_\bb}P_{\z_\aa^0}$ belongs to the $p$-Schatten class,  so does $P_{\k_\bb}M_{B_\aa}|_{\k_\bb}$, and therefore its diagonal entries, namely $\{B_\aa(b_j)\}$, belong to $\ell^p$.

One can exhibit examples of sequences $\aa$ , $\bb$ such that $P_{\k_\bb}P_{\z_\aa^0}$ is compact:
\begin{ejem}
 Consider $\bb$ an  interpolating sequence, and let $\aa$ such that $|b_n-a_n|$  tends fast enough to zero. For instance, consider $\displaystyle{f(t)=1-e^{-1/(t-1)^2}}$ and $a_n=f(|b_n|)b_n$. Then
$$
B_\bb(a_n)=\frac{b_n-a_n}{1-\bar{b}_na_n} \prod_ {k\ne n} \frac{b_k-a_n}{1-\bar{b}_ka_n}.
$$
The modulus of the first factor equals $\displaystyle{\frac{|b_n|(1-f(|b_n|)}{1-|b_n|^2f(|b_n|)}}$, which tends to zero as $n\to\infty$ (and $|b_n|\to 1$). The modulus of the second factor is less than one. 

In fact, for such $f$, it is easy to see that for $|b_n|$ near $1$,  
$\displaystyle{\frac{|b_n|(1-f(|b_n|))}{1-|b_n|^2f(|b_n|)}}\sim\displaystyle{\frac{1}{e^{1/(|b_n|-1)^2}(1-|b_n|)}}$,
 which tends to zero faster than any power of $(1-|b_n|)$, and then, in particular,  $P_{\k_\bb}P_{\z_\aa^0}$ is trace class
\end{ejem}

\end{rem}

\begin{rem}
In \cite{pqcompacto}, the compactness of the product $PQ$ of two orthogonal projections was studied. For instance, the following necessary and sufficient condition was established (Theorem {\bf 4.1} in \cite{pqcompacto}):

$PQ$ is compact if and only in there exist orthonormal bases $\{\xi_k: k\ge 1\}$ of $R(P)$ and $\{\psi_l: l\ge 1\}$ of $R(Q)$ such that
$$
\langle \xi_k,\psi_l\rangle=0 \hbox{ if } k\ne l , \hbox{ and } \langle \xi_k,\psi_k\rangle\to 0.
$$
Thus, in our case, $P_{\k_\bb}P_{\z_\aa^0}$ is compact if and only if there exist {\rm bi-orthogonal} bases $\{f_k\}$ of $\k_\bb$ and $g_l$ of $\z_\aa^0$ such that $\langle f_k,g_k\rangle\to 0$. Or equivalently, since $M_{B_\aa}:\k_\bb\to\z_\aa^0=\h_0\ominus\k_\aa$ is an (onto) isometry, there exist bases $\{f_k\}$ and $\{h_l\}$ of $\k_\bb$ such that $\langle f_k,B_\aa h_l\rangle=0$ if $k\ne l$ and $\langle f_k,B_\aa h_k\rangle\to 0$.
\end{rem}

\subsection*{Acknowledgment}
This research was supported by Grants CONICET (PIP 2016 0525), ANPCyT (2015 1505/ 2017 0883) and FCE-UNLP (11X829).

%\medskip

%\textit{E-mail address:} eandruch@ungs.edu.ar

%\textit{E-mail address:} eduardo@mate.unlp.edu.ar

%\textit{E-mail address:} avarela@ungs.edu.ar

\end{document}